\documentclass[11pt,oneside,reqno]{amsart}
      \usepackage{amssymb,enumerate,bbm}
\usepackage{marginnote}
\usepackage{amsmath}
\usepackage{mathtools}	
\usepackage[utf8]{inputenc}
\usepackage[english]{babel}

      \theoremstyle{plain}
      \newtheorem{theorem}{Theorem}[section]
      \newtheorem{lemma}[theorem]{Lemma}
      \newtheorem{corollary}[theorem]{Corollary}
      \theoremstyle{definition}
      \newtheorem{definition}[theorem]{Definition}   
      \theoremstyle{remark}
      \newtheorem{remark}[theorem]{Remark}		
	  \theoremstyle{proposition}
	  
       \newtheorem{assumption}[theorem]{Assumption}

      \makeatletter
      \def\@setcopyright{}
      
      \def\serieslogo@{}
      \makeatother

\usepackage{geometry}
\begin{document}

   \author{Kristina Schubert}
   \address{Institute of Mathematical Statistics,  University of M\"unster, 
             Orl\'{e}ans-Ring 10, 48149 Münster, Germany}
   \email{kristina.schubert@uni-muenster.de}

   \title[Spacings in Orthogonal and Symplectic Random Matrix Ensembles]{Spacings in Orthogonal and Symplectic Random Matrix Ensembles}

   \begin{abstract}
We consider the universality of the nearest neighbour eigenvalue spacing distribution in invariant random matrix ensembles. 
Focussing on orthogonal and symplectic invariant ensembles, we show that the  empirical  spacing distribution converges in a  uniform way. 
 More precisely, the main result states  that the expected Kolmogorov distance of the empirical spacing distribution from its universal limit converges to zero as the matrix size tends to infinity.
  \end{abstract}

   \subjclass[2010]{15B52}

   \keywords{random matrices, invariant ensembles, local eigenvalue statistics, nearest neighbour spacings, universality}

   \maketitle

\setcounter{equation}{0}
\section{Introduction}

The study of random matrices, starting with the works of Wishart \cite{Wish} in the 1920s, began to thrive in the 1950s due to the pioneering  ideas
of Wigner \cite{wigner1}, who suggested  to model energy levels in quantum systems using random matrices.
Ever since, a number of applications was discovered and the field revived in the late 1990s. This was initiated by the major discovery of Baik, Deift and Johansson \cite{TBaDeJo} that the limiting distributions of the length of the longest increasing subsequence of a
random permutation on $N$ letters and the largest eigenvalue of a random $N  \times N$ matrix in a certain model coincide.

In the theory of random matrices a considerable number of \textit{ensembles} has been studied over the last decades, where an ensemble is given by a set of matrices together with a probability measure. The most prominent and well-studied ensembles are the Gaussian ensembles (see e.g.~\cite{AGZ,mehta}), consisting of either real symmetric, complex Hermitian or quaternionic self-dual matrices with independent Gaussian entries (as far as the symmetry permits). These ensembles are characterised by a parameter $\beta$ that takes a  value in $\{1,2,4\}$ for each class of matrices.
 The classical Gaussian ensembles have been generalised in several ways leading e.g.~to the famous Wigner ensembles, where other distributions  than the normal distribution are allowed for the matrix entries. In the past years, local properties of Wigner ensembles were  rigorously studied  by Erd\H{o}s et al.~(see e.g.~\cite{Erdos2} for a survey) and Tao and Vu \cite{TaoVu} in a series of papers.
For a comprehensive overview on current developments in Random Matrix Theory the reader is referred to the recent books \cite{ABF,Forrester_log_gases}. 
 
Another generalisation of the Gaussian ensembles is motivated by the fact that in these ensembles the probability measure on the matrices is invariant under orthogonal, unitary and symplectic transformations respectively, justifying the full names Gaussian Orthogonal, Unitary and  Symplectic Ensembles. In fact, the Gaussian ensembles are, up to scaling, the only ensembles with both this invariance property and   independent matrix entries. Abandoning the requirement of independent matrix entries and only considering invariant measures on the symmetric, Hermitian or self-dual matrices leads to ensembles called  \textit{orthogonal ($\beta=1$), unitary ($\beta=2$)} and \textit{symplectic ($\beta=4$) invariant ensembles}, which we will study here (see e.g.~\cite{deift2} for a general introduction on invariant ensembles).

In Random Matrix Theory, besides global features of the spectrum, one is often interested in  local eigenvalue statistics such as the distribution of the largest eigenvalue or, as considered here, the distribution of spacings of adjacent eigenvalues. 
Of particular interest is the asymptotic behaviour of these quantities as the matrix size tends to infinity.  
The study of local eigenvalue statistics has been  driven by a phenomenon called \textit{universality}, that describes the fact that the limiting distributions of  several of these  statistics do not depend on the details of the ensemble but   on the symmetry class ($\beta=1,2,4$) only.
To obtain such universality results, typically, the first step in the analysis is to express the quantity of interest in terms of $k$-point correlations. 
One of the characteristic features of invariant ensembles is the fact that these $k$-point correlations can in turn be expressed in terms of orthogonal polynomials. Asymptotic results for the correlation functions can then be obtained  by the Fokas-Its-Kitaev  representation of orthogonal polynomials as a Riemann-Hilbert problem \cite{FIK} followed by the Deift-Zhou steepest descent method \cite{DeiftZhou}.
Hence, convergence of the $k$-point correlation functions has been shown  for a number of invariant ensembles. 

While in the literature this convergence is usually shown on compact sets, starting point for the proof of the Main Theorem of this paper is a slightly more refined universality assumption, stating the uniform convergence of the correlation functions on certain sets that grow with the matrix size. As this is what one would expect from the known techniques,  we do not derive such a refined universality result  here, but state it  as an assumption for the Main Theorem. 
Verifying the assumption for invariant ensembles even beyond the Hermite-type ensembles introduced here, makes our  Main Theorem applicable to other ensembles such as  Laguerre and Jacobi-type ensembles as well (see Remark \ref{Remark_Assumption}). 
We note that in this paper convergence of the correlation functions and all derived quantities is always considered in the bulk of the spectrum, i.e.~away from spectral edges, where the limiting correlation functions are determinants of the famous sine kernel. 

Our quantity of interest is the empirical distribution function of spacings of adjacent eigenvalues, often  denoted by \textit{spacing distribution}. 
Here, we consider localized and rescaled eigenvalues $\lambda_1(H) \leq \ldots \leq \lambda_N(H)$ of a $N\times N$ matrix $H$ (we may omit the $H$ in the notation and write $\lambda_i$ instead of $\lambda_i(H)$). This means that before rescaling we restrict our attention to eigenvalues $\widetilde{\lambda_i}(H)$ in intervals $I_N=[a-\delta_N,a+\delta_N]$, $\delta_N \to 0$ for some point $a$ in the bulk of the spectrum. Then we rescale the eigenvalues (afterwards denoted by $\lambda_i(H)$) such that  they have spacings that are close to one on average (details on the rescaling are given in section \ref{sec:2}). 
For finite $N$ the spacing distribution is then given by 
\begin{equation}\label{def_spacing_distr}
F_{N,\beta} (s, H) \coloneqq  \frac{1}{2 N  \delta_N} \# \{ \tilde{\lambda}_{i}(H), \tilde{\lambda}_{i+1}(H) \in I_N : \lambda_{i+1}-\lambda_i \leq s\}.
\end{equation}
Although $F_{N,\beta} 
$ can  be expressed in terms of $k$-point correlation functions, the assumed uniform convergence of the  correlation functions does not directly imply the kind of uniform convergence of the spacing distribution we aim at in our Main Theorem.  More precisely, we show that the expected Kolmogorov distance of the empirical spacing distribution from its universal limit $F_{\beta}$
converges to zero as the matrix size tends to infinity, i.e.
$$\lim_{N \to \infty }\mathbb{E}_{N,\beta} \left(\sup_{s \geq 0} \left|F_{N,\beta}(s) - F_{\beta}(s)\right|\right) =0, \quad \beta=1,2,4.$$ 

An essential step in the proof of the Main Theorem is the point-wise convergence of the expected empirical spacing distribution $\mathbb E_{N,\beta}(F_{N,\beta}(s)) \to F_{\beta}(s)$ for $N \to \infty$ (see Theorem \ref{pointwise} (ii)). For unitary invariant ensembles this was shown in \cite{deift1} and the result was extended to orthogonal and symplectic invariant ensembles in \cite{Alsm_Loewe}. We revisit the arguments of  \cite{Alsm_Loewe} in order to add some error estimates that are used to prove the uniform convergence stated in the Main Theorem.
Both our proof of this point-wise convergence and the rest of the proof of the Main Theorem are inspired by a proof of Katz and Sarnak in \cite{katzsarnak} for an ensemble related to the unitary invariant ensemble ($\beta=2$), called  
circular unitary ensemble. In particular, we can adapt their methods to express the empirical spacing distribution in terms of correlation functions and finally in terms of (matrix) kernel functions $K_{N,\beta}$. The uniform convergence of these  kernels $K_{N,\beta}$ in the region of interest is then assumed for the Main Theorem. 

 In order to prove the Main Theorem, we extend the methods  of \cite{katzsarnak}
 in two ways. On the one hand,  we need to
 introduce a localisation that  is not needed for the ensembles in
 \cite{katzsarnak}. On the other hand, 
  the relation between the kernel functions and the spacing distribution is
 way more involved for orthogonal and symplectic ensembles ($\beta=1$ and
 $\beta=4$) than it is for unitary ensembles. 
 Beyond that, one of the main obstacles in the proof of the Main Theorem are integrability issues concerning the matrix kernels $K_{N,\beta}$ (respectively their large $N$ limits) in the symplectic  case ($\beta=4$).
 
 Similar integrability issues were already encountered by Soshnikov  in \cite{soshnikov} for the circular $\beta=4$ case. 
The proofs in \cite{soshnikov} concentrate on circular unitary ensembles ($\beta=2$) and some of the results are extended to the circular $\beta=1$ case. 
The main result is that for the number of nearest neighbor  spacings smaller than $s$, denoted by $\eta_N(s)$, the process $(\eta(s) - \mathbb E \eta(s))N^{-1/2}$ weakly converges to some Gaussian random process on $[0,\infty).$ As a corollary the results of Katz and Sarnak  \cite{katzsarnak} are sharpened by improving the error estimate and by showing that the Kolmogorov distance of the spacing distribution   converges to zero with probability one while in \cite{katzsarnak}  the convergence of the expected Kolmogorov distance to zero is shown. The main ingredient of the proofs in \cite{soshnikov} is the consideration of the moments $\mathbb E (N^{-1/2}(\eta(s) - \mathbb E \eta(s)))^k$.  
However, as noted by Soshnikov this method does not seem to be applicable in the  $\beta=4$ case because of  the aforementioned integrability issues (see Section 7C in \cite{soshnikov} for an insightful comment on this problem).   
The   method introduced in this work to overcome these integrability issues in the current invariant setting might also be used to calculate $\mathbb E (N^{-1/2}(\eta(s) - \mathbb E \eta(s)))^2$ for circular $\beta=4$ ensembles even though its applicability to higher moments remains an open question.

 We proceed with a presentation of the setting and introduce the empirical
 spacing distribution in the next section. In Section 3 we study the $k$-point correlation functions and the  related matrix kernel functions $K_{N,\beta}$
 and present the Main Theorem together with the corresponding assumption on the kernels concerning their convergence. In
 Section 4 we introduce an auxiliary counting measure $\gamma_N$, which is used to show the pointwise convergence of the expected empirical spacing distribution. Section 5 is devoted to the most challenging estimate of the proof concerning the variance of the auxiliary measure $\gamma_N$. Finally, in Section 6 we combine the results of the previous sections  and complete the proof of the Main Theorem.
%
\section{Random Matrix Ensembles and the  Spacing Distribution}\label{sec:2}
We consider invariant random matrix ensembles with  $\beta =1,2,4$, which are
given by a set of $N \times N$ matrices, either real symmetric, complex Hermitian or quaternionic self-dual, together with a probability measure
$\mathbb{P}_{N,\beta}$. Depending on the set of matrices this probability measure is either invariant under orthogonal, unitary or
symplectic transformations. We focus on invariant Hermite-type ensembles with 
\begin{equation} \label{wmassinv}
\mathbb{P}_{N,\beta}(M)\,  dM = \frac{1}{Z_{N,\beta}}e^{- N \text{tr} (V(M))} \, dM, \quad \beta=1,2,4,
\end{equation}
where $dM$ denotes the Lebesgue measure on the considered matrices defined  as the product of Lebesgue-measures on algebraically independent entries of $M$. 
Moreover,  $Z_{N, \beta}$ denotes a normalisation constant and $V$ grows sufficiently fast such that all moments of $e^{-NV(x)}dx$ are finite. 
It is well known (see e.g.~\cite{deift2}) that with $w_\beta^{(N)}(x)=e^{- N V(x)}$ for $\beta=1,2$ and $ w_{\beta}^{(N)}(x) =e^{-2N  V(x)}$ for $\beta=4$
the joint probability density of the eigenvalues on the Weyl chamber $\mathcal{W}_N \coloneqq \{(\lambda_1,\ldots,\lambda_N) \in \mathbb R^N | \lambda_1 \leq \ldots \leq \lambda_N\}$ is given by 
\begin{equation}\label{jevinv}
d\mathbb{P}_{N,\beta}(\lambda_1, \ldots , \lambda_N)=\frac{1}{Z_{N,\beta}} \prod_{j<k}\left|\lambda_k-\lambda_j     \right|^\beta \, \, \prod_{j=1}^N w_{\beta}^{(N)}(\lambda_j) \,\, d\lambda_1 \ldots d\lambda_N,\quad \beta=1,2,4.
\end{equation}
In order to observe a limiting behaviour of the empirical spacing distribution, we need to introduce a proper rescaling,  which uses the notion of spectral density (see also \cite{Alsm_Loewe} for some further comments on the setting and the rescaling). Throughout this paper let  $\widetilde{\lambda}_1^{(N)}(H) \leq \widetilde{\lambda}_2^{(N)}(H) \leq \ldots \leq \widetilde{\lambda}_N^{(N)}(H)$ denote the ordered eigenvalues of a $N \times N$ matrix H. We may omit the superscript $N$ and the matrix $H$ in the notation $\widetilde{\lambda}_i^{(N)}(H) $ if those can be recovered from the context.
By the \textit{limiting spectral density} we denote a function $\psi$ such that for $s\geq 0$
\begin{equation}
\label{spec_density}
  \frac{1}{N}\mathbb{E}_{N,\beta}\left(\# \{j\colon \widetilde{\lambda}_j^{(N)} \leq s \} \right) \to \int_{-\infty}^s \psi(t) dt, \quad N \to \infty, 
\end{equation}
where $\mathbb{E}_{N,\beta}$ is the expectation with respect to $d\mathbb{P}_{N,\beta}$.
The  existence of the spectral density was shown for large classes of invariant ensembles and hence we also assume its existence for the ensembles considered here.
Let $a$ denote a point of continuity of $\psi$ in the interior of the support of $\psi$ with $\psi(a)>0$.  
In a small neighbourhood of $a$, where $\psi$ is close to $\psi(a)$, the expected distance of adjacent eigenvalues $\widetilde{\lambda}_i$ is to leading order given by $(N \psi(a))^{-1}$. Hence, we introduce the rescaled and centred eigenvalues
\begin{equation}\label{rescale}
		\lambda_i\coloneqq  ( \widetilde{\lambda}_i -a) N\psi(a).
\end{equation}
We restrict our attention to eigenvalues $\widetilde{\lambda}_i$ that lie in a  $N$-dependent interval $I_N$  with vanishing length
$|I_N| \to 0$ for $N \to \infty$ and centre $a$. By (\ref{spec_density}) the spacings of the  rescaled eigenvalues $\lambda_i$ are close to one on average.
We further denote the rescaled counterparts of the intervals $I_N$ by
	\begin{equation*} 
		A_N\coloneqq  N \psi(a)(I_N-a)=\{N \psi(a)(t-a) \mid t \in I_N \},
	\end{equation*}
i.e.~$\widetilde{\lambda}_i \in I_N$ if and only if $\lambda_i \in A_N$. 
Moreover, we assume   $N|I_N| \to \infty$ for $N \to \infty$, which ensures that $|A_N| \to \infty$ as $N \to \infty$ and hence a growing number of eigenvalues is taken into account as the matrix size tends to infinity. 
\begin{remark}\label{remark_setting}
In the remainder of this paper we consider $V$ in (\ref{wmassinv}), $a$ and $I_N$ as fixed and we assume that the spectral density $\psi$ exists.
Throughout the proofs $C$ denotes a generic constant, which may differ from line to
line, but does not depend on any parameters of the proof except  $V$, $a$ and
$I_N$.
Whenever we use the $\mathcal{O}$-notation, the constant implicit in this notation is also uniform in all parameters of the proof.
In some estimates we will use the notation
\begin{equation*}
\overline{\alpha} \coloneqq \max(1,\alpha), \quad \alpha>0.
\end{equation*} 
\end{remark}
With this notation we can now define the counting measure of the spacings of adjacent eigenvalues $\sigma_N$ and hence  the empirical spacing distribution  $\int_0^s d\sigma_N$. 
\begin{definition}
For a matrix H with eigenvalues $\widetilde{\lambda}_1 \leq \ldots \leq
\widetilde{\lambda}_N$  and rescaled eigenvalues  $\lambda_1 \leq \ldots \leq
\lambda_N$ (see (\ref{rescale})) we set
\begin{equation*}
\sigma_N(H) \coloneqq \frac{1}{|A_N|} \sum_{\widetilde{\lambda}_{i+1},\widetilde{\lambda}_i \in I_N} \delta_{\lambda_{i+1}-\lambda_{i}}.
\end{equation*}
\end{definition}
We observe that in the definition of $\sigma_N$  instead of normalizing by the actual number of spacings or eigenvalues in $I_N$, which depends on $H$ and is hence random, we rather normalise by the deterministic factor $|A_N|$  since we expect about $|A_N|$ eigenvalues   in $I_N$.
Our Main Theorem then states that there exist   functions  $F_{\beta}$ with $\beta \in \{1,2,4\}$
such that for large classes of potentials $V$ the expected Kolmogorov distance
of the empirical spacing distribution from $F_{\beta}$ converges
to zero as the matrix size tends to infinity, i.e. 
\begin{equation*} 
\lim_{N \to \infty }\mathbb{E}_{N,\beta} \left(\sup_{s \geq 0} \left|\int_{0}^s \, 
   d\sigma_N(H) - F_{\beta}(s)\right|\right) =0, \quad \beta=1,2,4.
\end{equation*}
Later (see Assumption \ref{main_asssumption}), the hypothesis of the theorem will not be formulated as an explicit assumption on $V$, but rather on certain  matrix kernels $K_{N,\beta}$ . These matrix kernels appear in the context of the correlation functions $  R_{N,k}^{(\beta)}$, which are an important tool to analyse several local eigenvalue statistics.
We note that the Main Theorem and all intermediate lemmas are valid for $\beta=1,2,4$, but we focus on the more complicated cases $\beta=1$ and $\beta=4$. 
%
\section{Correlation Functions and the Main Theorem}
Before we can give a precise statement of our Main Theorem, and in particular the required assumption, we start with the definition of the correlation functions  and their rescaled counterparts, which leads to the introduction of the (matrix) kernels $K_{N,\beta}$.
\begin{definition}\label{defR}
	Let  $R_{N,N}^{(\beta)}$ denote the density function on the right hand side of (\ref{jevinv}), i.e.~$d\mathbb{P}_{N,\beta}(\widetilde{\lambda}_1, \ldots , \widetilde{\lambda}_N)=R_{N,N}^{(\beta)} (\widetilde{\lambda}_1, \ldots , \widetilde{\lambda}_N) \,\, d\widetilde{\lambda}_1 \ldots d\widetilde{\lambda}_N$, and  consider $R_{N,N}^{(\beta)}$ as a function on $\mathbb R^N$.
For $k \in \mathbb{N}, k\leq N$ and $\beta\in \{1,2,4\}$ we set 
\begin{align*} 
 R_{N,k}^{(\beta)}(\widetilde{\lambda}_1, \ldots, \widetilde{\lambda}_k)&\coloneqq\frac{1}{(N-k)!}   
\int_{\mathbb{R}^{N-k}} 
R_{N,N}^{(\beta)}(\widetilde{\lambda}_1,\ldots, \widetilde{\lambda}_N) \, 
d\widetilde{\lambda}_{k+1}\ldots d\widetilde{\lambda}_N,
\\
 B_{N,k}^{(\beta)}(\lambda_1, \ldots, \lambda_k)& \coloneqq \left( N \psi(a) \right)^{-k} 
 R_{N,k}^{(\beta)} \left(\widetilde{\lambda}_1, \ldots,\widetilde{\lambda}_k \right). \nonumber
\end{align*}
\end{definition}
\begin{remark}\label{intout}
		In the later proofs we will often use the fact that     $R_{N,k}^{(\beta)}(t_1,
		\ldots,t_k)$ and  $B_{N,k}^{(\beta)}(t_1,
		\ldots,t_k)$  are  invariant under permutations of the 
 		indices $\{1, \ldots, k \}$, which can easily be seen from the definitions.	
\end{remark}
The rescaled correlation functions $B_{N,k}^{(\beta)}$ can be expressed in terms of a kernel function $\hat{K}_{N,2}: \mathbb R^2 \to \mathbb R$ for $\beta=2$ and in terms of matrix kernel functions $\hat{K}_{N,\beta}: \mathbb R^2  \to \mathbb R^{2 \times 2}$ for $\beta=1,4$ as follows (see \cite{mehta_mahoux}, \cite{TracyWidom1}): 
\begin{equation}
\label{DEt_Darst_B}
B_{N,k}^{(2)}(t_1,
		\ldots,t_k)=\det \left( \hat{K}_{N,2}(\tilde{t_i}, \tilde{t_j})\right)_{1\leq i,j\leq k},
\end{equation}
\begin{equation*}
B_{N,k}^{(\beta)}(t_1,
		\ldots,t_k)=\text{Pf}(\hat{K}_{N,\beta}(\tilde{t_i},\tilde{t_j})_{1\leq i,j\leq k}  J), \quad \beta=1,4,
\end{equation*}
\begin{equation}	
\label{Def_J}	
J\coloneqq\text{diag}(\sigma,\ldots,\sigma) \in \mathbb R^{2k \times 2k}, \quad \sigma\coloneqq\begin{pmatrix*}[r] 0 & \phantom{-}1 \\ -1 & 0\end{pmatrix*},
\end{equation}
where Pf denotes the Pfaffian of a matrix.
These kernel functions $\hat{K}_{N,\beta}$ are related to orthogonal
polynomials, permitting an asymptotic analysis as the matrix size tends to
infinity. For many invariant ensembles  we have the convergence to the sine kernel for $\beta=2$ (see e.g.~\cite{deift1}, \cite{DeiftKriecherbauer}, \cite{Lev_Lub})
\begin{equation}
\label{sine_kernel}
\lim_{N \to \infty}\hat{K}_{N,2}(\tilde{x}, \tilde{y}) = \frac{\sin(\pi (x-y))}{\pi (x-y)} \eqqcolon K_2(x,y)
\end{equation}
and for $\beta=1,4$ (see \cite{deift2}, \cite{Scherbina})
\begin{equation*}
\hat{K}_{N,\beta}(\tilde{x},\tilde{y}) \xrightarrow{N\to\infty} \begin{pmatrix} \mathsf{S}_{\beta}(x,y) & \mathsf{D}_{\beta}(x,y) \\ \mathsf{I}_{\beta}(x,y) & \mathsf{S}_{\beta}(x,y) \end{pmatrix}=: K_{\beta}(x,y)
\end{equation*} 
with 
\begin{align}
\mathsf{S}_{1}(x,y) &\coloneqq K_{2}(x,y) , \quad  \mathsf{S}_{4}(x,y) \coloneqq K_{2}(2x,2y),	 
\label{notation_S}
\\
\mathsf{D}_{1}(x,y)& \coloneqq \frac{\partial }{\partial x}K_{2}(x,y), \quad    
\mathsf{D}_{4}(x,y)\coloneqq \frac{\partial }{\partial x}K_{2}(2x,2y),
\label{notation_D}
\\
\mathsf{I}_{1}(x,y) &\coloneqq \int_0^{x-y} K_{2}(t,0)dt -\frac{1}{2}\text{sgn}(x-y),\quad   
 \mathsf{I}_{4}(x,y) \coloneqq \int_0^{x-y} K_{2}(2t,0)\, dt.
\label{notation_I}
\end{align}
Although the convergence $\hat{K}_{N,\beta} \to K_{\beta}$ for $N \to \infty$ has been established in
quite some generality, we need a slightly more refined result for the formulation of the Main Theorem. 
More precisely, we need the convergence to be uniform on the growing sets $A_N$, while in the literature it is commonly shown for compact sets.
 We formulate this  uniform
 convergence of $\hat{K}_{N,\beta}$, which is what one would
 expect from known universality results,  as an assumption.

\begin{assumption}\label{main_asssumption}
In addition to the setting described in the Introduction (see Remark
\ref{remark_setting}), we assume that $A_N$, $a$ and $V$ are given such that the following statement is true:
There exists a sequence $(\kappa_N)_{N\in \mathbb{N}}$ with $\lim_{N \to \infty} \kappa_N=0$
 such that  for $\beta=1,2,4$ and $x,y \in A_N,$ $\tilde{x}=a+\frac{x}{N\psi(a)},$ $  \tilde{y}= a+\frac{y}{N\psi(a)}$  we have 
\begin{equation}
\label{conv_kernel}
\hat{K}_{N,\beta}\left(\tilde{x},\tilde{y}\right)=K_{\beta}(x,y)+ \mathcal{O}(\kappa_N), 
\end{equation}
where the constant implicit in the error term is uniform for all $x,y \in A_N$.
\end{assumption} 
 We comment on the scope of Assumption \ref{main_asssumption} after stating the Main Theorem. 
%
\begin{theorem}[Main Theorem]
\label{MAINTHEO}
There exist probability measures $\mu_{\beta}, \beta=1,2,4$ such that under Assumption \ref{main_asssumption} we have
\begin{equation*}
\lim_{N \to \infty }\mathbb{E}_{N,\beta} \left(\sup_{s \geq 0} \left|\int_{0}^s \, 
   d\sigma_N(H) -\int_0^s d\mu_{\beta}\right|\right) =0.
\end{equation*}
\end{theorem}
In the following remark we argue that the assumption of the Main Theorem is actually satisfied by a certain class of potentials $V$ for Hermite-type ensembles and discuss its validity beyond these ensembles.
\begin{remark}
\label{Remark_Assumption}
A general  proof of Assumption \ref{main_asssumption} is not readily available in the literature, especially not for $\beta=1$ and $\beta=4$.  For Hermite-type ensembles with varying weights it can in fact easily, though technically,  be deduced from known results for a certain class of potentials $V$. The assumption can further be expected to hold for a number of invariant random matrix ensembles.
\begin{enumerate}
\item[(a)]\textit{Hermite-type ensembles with varying weights:} Assumption \ref{main_asssumption} can be verified for the ensembles given in  (\ref{wmassinv}) with  $|A_N|/\sqrt{N} \to 0$ as $N \to \infty$ and  $\kappa_N=\mathcal{O}(|A_N|/\sqrt{N})$ if the potential $V$ satisfies: 
\begin{itemize}
\item $V$ is a polynomial of even degree with positive leading coefficient and
\item $V$ is regular in the sense of \cite[(1.12), (1.13)]{DeiftKriecherbauer}. In particular, each convex potential $V$ satisfies this regularity condition.
\end{itemize}

The first step in the proof is to show the following result for the $\beta=2$ kernel $\hat{K}_{N,2}$: 
For $V$   real analytic, regular in the sense of \cite[(1.12), (1.13)]{DeiftKriecherbauer}, $\lim_{|x| \to \infty} \frac{V(x)}{\ln (x^2+1)}=\infty$
and 
$d_N(x,y)\coloneqq \hat{K}_{N,2}\left( \tilde{x},\tilde{y}\right)  - K_2(x,y)$
we have
$$\sup_{x,y \in A_N} |d_N(x,y)|=\mathcal{O}\left(|A_N|N^{-1}\right), \quad \sup_{x,y \in A_N} \left|\frac{\partial}{ \partial y}d_N(x,y)\right|=\mathcal{O}\left(|A_N|N^{-1}\right) $$
and 
$$\sup_{x,y \in A_N} \int_x^y |d_N(s,y)| ds=\mathcal{O}\left(|A_N|^2N^{-1}\right). $$
This result can basically be obtained by  adjusting the Riemann-Hilbert analysis presented in \cite{DeiftKriecherbauer} together with some ideas of \cite{Vanlessen} (c.f.~\cite[Theorem 3.4]{diss} for a sketch of the proof).

The main ingredient for the next step of the proof is the well-known fact  that the entries of the matrix $\hat{K}_{N,\beta}\left(\tilde{x},\tilde{y}\right)$ can all be expressed in terms of its $(1,1)$-entry as follows: 
For $\beta=1$ and $\beta=4$ we have 
\begin{align*}
(\hat{K}_{N,\beta}\left(\tilde{x},\tilde{y}\right))_{2,2} 
&= 
(\hat{K}_{N,\beta}\left(\tilde{y},\tilde{x}\right))_{1,1}, \quad 
(\hat{K}_{N,\beta}\left(\tilde{x},\tilde{y}\right))_{1,2}
=
\frac{\partial}{\partial y}(\hat{K}_{N,\beta}\left(\tilde{x},\tilde{y}\right))_{1,1}
\\
(\hat{K}_{N,4}\left(\tilde{x},\tilde{y}\right))_{2,1} 
&=
 - \int_{\tilde{x}}^{\tilde{y}} (\hat{K}_{N,4}\left(t,\tilde{y}\right))_{1,1} \, dt,
 \\ 
 (\hat{K}_{N,1}\left(\tilde{x},\tilde{y}\right))_{2,1} &= - \int_{\tilde{x}}^{\tilde{y}} (\hat{K}_{N,1}\left(t,\tilde{y}\right))_{1,1} \, dt - \frac{1}{2} \text{sgn}(\tilde{x} - \tilde{y}).
 \end{align*} 
One can then show that for  monomials $V(x)=x^{2m}$ that are regular in the sense of \cite[(1.12), (1.13)]{DeiftKriecherbauer} and $c_N (x,y)\coloneqq(\hat{K}_{N,\beta}\left(\tilde{x},\tilde{y}\right))_{1,1} -\hat{K}_{N,2}\left( \tilde{x},\tilde{y}\right)$ we have
$$\sup_{x,y \in A_N} |c_N(x,y)|=\mathcal{O}\left(N^{-\frac{1}{2}} \right), \quad \sup_{x,y \in A_N} \left|\frac{\partial}{ \partial y}c_N(x,y)\right|=\mathcal{O}\left(N^{-\frac{1}{2}} \right)  
$$
and 
$$
\sup_{x,y \in A_N} \int_x^y |c_N(s,y)| ds=\mathcal{O}\left(|A_N|N^{-\frac{1}{2}} \right). $$
We can derive these estimates by extending those obtained from the Riemann-Hilbert analysis in \cite{deift2} (relying on the methods introduced by Widom \cite{Widom}) to uniform estimates in the growing sets $A_N$. A generalization from monomials $V$ to polynomials of even degree can then be obtained from \cite{Scherbina}.

In the special case of $\beta=2$ Assumption \ref{main_asssumption} is verified for real analytic $V$ with a strictly increasing derivative $V'$ and $\lim_{|x| \to \infty} V(x)=\infty$ in  \cite{KSSV}.
The work of McLaughlin and Miller \cite{McL_M} suggests that for $\beta=2$ it is even sufficient that $V$
possesses two Lipschitz continuous derivatives (together with some assumptions on the equilibrium measure) instead of requiring $V$ to be analytic.
\item[(b)]\textit{Other ensembles:}
Besides Hermite-type invariant ensembles, a number of other invariant ensembles exists, for which convergence of the rescaled kernel $\hat{K}_{N,\beta}$ to $K_{\beta}$  was obtained by Riemann-Hilbert analysis. Due to the 	resemblance of the Riemann-Hilbert analysis to the case of varying Hermite weights, we can expect Assumption \ref{main_asssumption} to hold for such ensembles. 
Among those ensembles are e.g.~
\begin{itemize}
\item Laguerre-type ensembles  \cite{DGKV}, where $w_{\beta}^{(N)}$ in (\ref{jevinv}) is given by 
$$w_{\beta}(x)= w_{\beta}^{(N)}(x)= \begin{cases} x^\gamma e^{-Q(x)}, & \beta=1,2 \\  x^{2\gamma} e^{-2Q(x)}, & \beta=4  \end{cases}$$
with $\gamma >0$ and $Q$
denotes a polynomial of positive degree and with positive leading coefficient. 
\item modified Jacobi unitary ensembles ($\beta=2$) \cite{KuV2}, where
$$w_2(x)=(1-x)^{\alpha_1} (1+x)^{\alpha_2} h(x), \quad x \in (-1,1)$$
with $\alpha_1,\alpha_2 > -1$ and $h$   real valued, analytic and taking only positive values on $[-1,1]$. 
\item  non-varying Hermite-type ensembles \cite{DKMVZ1,deift2}, where 
$$w_{\beta}(x)=e^{-Q(x)},\quad \beta=1,2,4$$
with a polynomial $Q$ of even degree with positive leading coefficient.
\end{itemize}
\end{enumerate}
\end{remark}
Before  we start to prove the Main Theorem, we further consider the (rescaled) correlation functions and their limits. 
In order to state some results about $B_{N,k}^{(\beta)}$, we denote their large $N$ limits, which exist under Assumption \ref{main_asssumption}, by
\begin{equation*}
W_k^{(\beta)} (t_1,\ldots, t_k) \coloneqq \lim_{N \to \infty }B_{N,k}^{(\beta)}(t_1,
		\ldots,t_k), \quad \beta=1,2,4.
\end{equation*}
In order to expand the correlation functions and their limits in terms of the kernel functions $\hat{K}_{N,\beta}$ and $K_{\beta}$ for $\beta=1$ and $\beta=4$,
we introduce some more notation. 

Let  $S=\{i_1,\ldots,i_k\} \subset \{1,\ldots, n\}, i_1 < \ldots < i_k$ denote a set of cardinality $k$. For a vector $x=(x_1,\ldots,x_n) \in\mathbb R^n$  we set 
\begin{equation}
\label{notation_vector}
x_S\coloneqq (x_{i_1}, \ldots, x_{i_k})
\end{equation}
and we denote the set of bijections on $S$ by
\begin{equation*}
\mathfrak{S}(S) \coloneqq \{  \sigma \colon S  \to S \mid  \sigma \text{ is a bijection}\}.
\end{equation*}
For a matrix kernel $K: \mathbb R^2 \to \mathbb R^{2 \times 2}$, $\sigma \in \mathfrak{S}(S)$, $d=(d_1,\ldots,d_k)\in \mathbb R^k, t=(t_1,\ldots,t_N) \in \mathbb R^N$ we set 
\begin{equation}
\label{Def_F_S}
F_{S,\sigma}^K (d, t_S) 
\coloneqq \frac{1}{2k} (K(t_{\sigma(i_1)},t_{\sigma(i_2)}))_{d_1 d_2} (K(t_{\sigma(i_2)},t_{\sigma(i_3)}))_{d_2 d_3} \ldots (K(t_{\sigma(i_k)},t_{\sigma(i_1)}))_{d_kd_1}.
\end{equation}
With this notation the correlation functions can be expressed by 
(see \cite{TracyWidom1})
\begin{equation}
\label{expansion_corr}
B_{N,k}^{(\beta)} (t_1,\ldots,t_k) = \sum_{m=1}^k \sum_{S_1, \ldots, S_m}
(-1)^{k-m}  \sum_{\substack{\sigma_1 \in \mathfrak S(S_1),\\\ldots,\\\sigma_m \in \mathfrak S(S_m)}}\ \sum_{d_1,\ldots,d_k=1}^2 \prod_{i=1}^m  F_{S_i,\sigma_i}^{\hat{K}_{N,\beta}}(d_{S_i},\tilde{t}_{S_i}), \quad \beta=1,4,
\end{equation}
and hence
\begin{equation}
\label{expansion_corr_limit}
W_k^{(\beta)} (t_1,\ldots,t_k) = \sum_{m=1}^k \sum_{S_1, \ldots, S_m}
(-1)^{k-m}  \sum_{\substack{\sigma_1 \in \mathfrak S(S_1),\\\ldots,\\\sigma_m \in \mathfrak S(S_m)}}\ \sum_{d_1,\ldots,d_k=1}^2 \prod_{i=1}^m  F_{S_i,\sigma_i}^{K_{\beta}}(d_{S_i},t_{S_i}), \quad \beta=1,4,
\end{equation}
where for each $m$ we sum over all partitions of $\{1,\ldots, k\}$ into
non-empty subsets $S_1,\ldots, S_m$. Recall that  $S_1,\ldots,S_m$ and $S_{\pi(1)},\ldots,S_{\pi(m)}$ denote the same partition for any permutation $\pi$.

 We observe further that each entry of the limiting kernels $K_{\beta}$ in (\ref{notation_S}) to (\ref{notation_I}) is a bounded function and by the uniform convergence required by Assumption \ref{main_asssumption} we have for some $C>0$
 \begin{equation*}
 \sup_{N \in \mathbb N} \sup_{x,y \in A_N} \left| (\hat{K}_{N,\beta}(\widetilde{x},\widetilde{y}))_{ij} \right| \leq C, \quad i,j=1,2,\quad \beta=1,4.
 \end{equation*}
For later reference, we note that  the general identity 
\begin{equation*}
\prod_{i=1}^k a_i - \prod_{i=1}^k b_i=  \sum_{j=1}^k \left(\prod_{i=1}^{j-1} a_i \right)(a_j-b_j) \left(\prod_{i=j+1}^k b_i\right)
\end{equation*} 
implies (under Assumption \ref{main_asssumption}) 
\begin{equation}
\label{est_F.1}
\left|\prod_{i=1}^m  F_{S_i,\sigma_i}^{\hat{K}_{N,\beta}}(d_{S_i},\tilde{t}_{S_i})-\prod_{i=1}^m  F_{S_i,\sigma_i}^{K_{\beta}}(d_{S_i},t_{S_i})\right| \leq C^m \mathcal{O}(\kappa_N).
\end{equation}
With these estimates and notation we can prove the following results about the correlation functions and their convergence under Assumption \ref{main_asssumption}. 
\begin{lemma}
\label{result_corr_func}
 For $k,N
\in\mathbb N$, $k\leq N$ and $\beta=1,2,4$ we have under  Assumption \ref{main_asssumption}:
\begin{enumerate}
\item[(i)]
$ \displaystyle B_{N,k}^{(\beta)}(t_1,\ldots,t_k)=W_k^{(\beta)}(t_1,\ldots,t_k)
+ k!  C^k \mathcal{O}(\kappa_N), \quad t_1,\ldots, t_k \in A_N. $
\\
\item[(ii)]
$W_k^{(\beta)}$ is   symmetric and additive translational invariant   on $\mathbb R^k$.
\\
\item[(iii)]
 $ \displaystyle |B_{N,k}^{(\beta)}(t_1,\ldots,t_k)| \leq C^k k^{k/2}, \quad |W_k^{(\beta)}(t_1,\ldots,t_k)| \leq C^k k^{k/2}, \quad t_1,\ldots,t_k \in A_N. 
 $
\end{enumerate}
\end{lemma}
\begin{proof}
We use the basic estimate 
\begin{equation} \label{absch_kombi_1}
\sum_{m=1}^{k}
		\sum_{S_1,\ldots, S_m} \left( \prod_{i=1}^m |S_i|! \right) \leq 4^{k} k!,
\end{equation}
where $S_1,\ldots, S_m$ denotes a partition of $\{1,\ldots,k\}$.
The proof of (i) for $\beta=1$ and $\beta=4$ is then obvious from (\ref{est_F.1}). For $\beta=2$ statement (i) follows from the  Leibniz representation of $B_{N,k}^{(\beta)}$ and
$W_k^{(\beta)}$. 
 Statement (ii) is obvious from  (\ref{sine_kernel}) and (\ref{notation_S})  to (\ref{notation_I}). It remains to show the upper bound in (iii) for $W_k^{(\beta)}$ as the respective bound for $B_{N,k}^{(\beta)}$ follows from the uniform convergence.
We observe that  for an arbitrary positive-semidefinite matrix $K \in \mathbb R^{n \times n}$ the inequality of the geometric and arithmetic mean, applied to the eigenvalues, gives  
\begin{equation}
\det(K) \leq   \frac{1}{n^n} \left( \text{tr}(K) \right)^n \leq \frac{1}{n^n} (Cn)^n=C^n
\end{equation}
if all entries of $K$ are bounded by $C$, i.e.~$|K_{ij}|\leq C$ for $i,j=1,\ldots,n$.
Hence, we have for an arbitrary $n\times n$ matrix $K$ with entries bounded by $C$
\begin{equation}
\label{bound_K}
\det(KK^\text{T}) \leq  n^n C^{2n}.
\end{equation}
To prove the bound on $W_k^{(\beta)}$ we set for $t=(t_1,\ldots,t_k)$
$$ \mathcal{K}_{\beta}(t):= \left( K_{\beta}(t_i, t_j)\right)_{1\leq i,j\leq k},$$
where $ \mathcal{K}_{\beta}(t) \in\mathbb R^{k \times k}$ for $\beta=2$ and $ \mathcal{K}_{\beta}(t) \in\mathbb R^{2k \times 2k}$ for $\beta=1,4$.
The boundedness of $K_{\beta}(x,y)$ and (\ref{bound_K}) then imply for all $t_1,\ldots,t_k \in \mathbb R$
\begin{equation*}
\left|W_{k}^{(2)}(t_1,
		\ldots,t_k)\right|=\left|\det \left(  \mathcal{K}_{2}(t)\right)\right|= \sqrt{\det(\mathcal{K}_{2}(t)\mathcal{K}_{2}(t)^\text{T})} \leq k^{k/2} C^k 
\end{equation*}
and for $\beta=1,4$ (recall $n=2k$ in (\ref{bound_K}))
\begin{equation*}
 \left|W_{k}^{(\beta)}(t_1,
		\ldots,t_k) \right|= \text{Pf}(\mathcal{K}_{\beta}(t)  J)=\sqrt{\det \left(\mathcal{K}_{\beta}(t)  J \right)}= \det \left(\mathcal{K}_{\beta}(t)  \mathcal{K}_{\beta}(t)^\text{T}   \right)^{\frac{1}{4}}\leq k^{k/2} C^k.
\end{equation*}
\end{proof}
Another important feature of the entries of the limiting kernels $K_{\beta}$, which we need to consider, is their  integrability. 
Since, by some easy calculation, we have $\mathsf{I}_1(x,y)= \mathcal{O}(1/|x-y|)$, 
the square integrability 
\begin{equation}
\label{eq_square_int.1}
\int_{\mathbb R} (f(x,y))^2 \, dx \leq C \quad \forall y\in \mathbb R
\end{equation}
is satisfied for $f\in\{\mathsf{S}_1,\mathsf{D}_1,\mathsf{I}_1,\mathsf{S}_4,\mathsf{D}_4\}$. It is the lack of square integrability for $\mathsf{I}_4$ that complicates the later proof in the case $\beta=4$. 
This will be compensated by the following two observations:
\begin{enumerate}
\item[(i)] There exists some $C>0$ such that  for any interval $\mathcal{A} \subset \mathbb R$ we have:  
\begin{equation}
\label{eq_int.1}
\left|  \int_{\mathcal{A}} f(x,y) dx\right| \leq C \quad \forall y \in \mathbb R, \quad f \in \{ \mathsf{S}_4,\mathsf{D}_4\}. 
\end{equation} 
This can be shown by direct integration and the boundedness of $\mathsf{S}_4, \mathsf{I}_4$.
\item[(ii)] There exists some $C>0$ such that  for any $\alpha >0$, $u,v \in [-\alpha,\alpha]$, $y \in \mathbb R$ and 
 \begin{equation}
\label{Def_R}	       
\mathcal{R}_{u,v}(x,y)\coloneqq     
	       \begin{cases}  \mathsf{I}_4(x+u,y+v)- \frac{1}{4}, &\text{if }  y<x\\\mathsf{I}_4(x+u,y+v)+ \frac{1}{4}, & \text{if }y 
	       \geq x
	       \end{cases}
\end{equation}
we have
\begin{equation}
\label{int_R}
\int_{\mathbb R} \left(\mathcal{R}_{u,v} (x,y)\right)^2\, dx \leq  C  \overline{\alpha}, \quad \quad  \left| \int_\mathcal{A}   \mathcal{R}_{u,v} (x,y)\, dx \right| \leq C \overline{\alpha}.
\end{equation}
These estimates can be verified for $u=v=0$  by some easy calculation, where  for the second inequality one can use integration by parts. The
auxiliary dependence on $\alpha$ is caused by the fact that the $x$- and $y$-arguments of $\mathsf{I}_4$ in (\ref{Def_R}) are shifted by at most $\pm\alpha$.
\end{enumerate}
%
%
\section{The Expected Empirical Spacing Distribution and its Limit}\label{Sec:pointwise}
In this section we consider the expected empirical spacing distribution, denoted by $ \mathbb E_{N,\beta} \left(  \int_0^s d\sigma_N(H)\right)$, and show the  existence of its large $N$ limit. 
Moreover, we derive a representation of this limit. For the convenience of the reader, we present some of the calculations in \cite{Alsm_Loewe}, adding more refined error estimates. Moreover, we comment on the connection between spacings and gap probabilities and present a further estimate for the proof of the Main Theorem in Corollary \ref{koro_ew}.

We observe that with  
$$ f(\lambda_1^{(N)}(H), \ldots, \lambda_N^{(N)}(H)) \coloneqq \int_0^s d\sigma_N(H), \quad \quad (\lambda_1^{(N)}(H), \ldots, \lambda_N^{(N)}(H)) \in \mathcal{W}_N$$
we can calculate the expected empirical spacing distribution  by 
$$ \mathbb E_{N,\beta} \left(  \int_0^s d\sigma_N(H)\right) = \int_{\mathcal{W}_N} f(t_1,\ldots,t_N) B_{N,N}^{(\beta)}(t_1,\ldots,t_N)dt.$$
Unfortunately, in this representation we cannot use the invariance property of the correlation functions mentioned in Remark \ref{intout}, as $f$ can only be defined for ordered $N$-tuples. 
Therefore, we 
follow the ideas of \cite{katzsarnak} and introduce  counting measures $\gamma_N(k,H)$ for $k \geq 2$ that are
closely related to $\sigma_N$.
For a random matrix $H$ with eigenvalues $\widetilde{\lambda}_1 \leq  \ldots
\leq \widetilde{\lambda}_N$ and rescaled eigenvalues $\lambda_1 \leq  \ldots
\leq \lambda_N$ we set for $k\geq 2$
\begin{equation*}
	\gamma_N(k,H)\coloneqq  \frac{1}{|A_N|}
	\sum_{\substack{i_1 < \ldots < i_k,\\ \lambda_{i_1},\lambda_{i_k}\in 
	A_N}}\delta_{(\lambda_{i_{k}}-
	\lambda_{i_1})}
	=  \frac{1}{|A_N|}
	\sum_{\substack{i_1 < \ldots < i_k,\\ \lambda_{i_1},\lambda_{i_k}\in 
	A_N}}\delta_{(\max_{1 \leq j \leq k} \lambda_{i_j} - \min_{1 \leq j \leq k} \lambda_{i_j})}.
\end{equation*}
The measures $\gamma_N(k,H)$ are related to $\sigma_N(H)$ as presented in Lemma
\ref{lemma_combinatoric} below. We will use that 
$\int_0^s
d\gamma_N(k,H)$  is a symmetric function of the eigenvalues of $H$, i.e.~with 
$$ g(\lambda_1^{(N)}(H), \ldots, \lambda_N^{(N)}(H)) \coloneqq \int_0^s d\gamma_N(k,H)$$
we have (see Remark \ref{intout})
\begin{align*}
\mathbb E_{N,\beta} \left(  \int_0^s d\gamma_N(k,H)\right) &= \int_{\mathcal{W}_N} g(t_1,\ldots,t_N) B_{N,N}^{(\beta)}(t_1,\ldots,t_N)dt
\\
&= \frac{1}{N!}\int_{\mathbb{R}^N} g(t_1,\ldots,t_N) B_{N,N}^{(\beta)}(t_1,\ldots,t_N)dt.
\end{align*}
With the notation of (\ref{notation_vector}) and 
\begin{equation*}
	\chi_s(t_t,\ldots,t_k)\coloneqq \chi_{(0,s)} \left(\max_{i=1,\ldots,k} t_i 	-\min_{i=1,\ldots,k} t_i \right) \prod_{i=1}^k \chi_{A_N} (t_i),
	\end{equation*}
where $\chi_{(0,s)}$ and $\chi_{A_N}$ denote the characteristic functions on $(0,s)$ and on $A_N$ respectively
we can further calculate
\begin{align}
		\mathbb E_{N,\beta} \left(\int_0^s d\gamma_N(k,H) \right)
		=& \frac{1}{N!} \int_{\mathbb R^N} \left( \frac{1}{|A_N|} \sum_{T 
		\subset 
		\{1,\ldots, N\}, |T|=k} \chi_s (t_T) \right)
		B_{N,N}^{(\beta)}(t) dt 
		\nonumber
		\\
		=& \frac{1}{|A_N|} \frac{1}{N!} \binom{N}{k} \int_{\mathbb R^N}    
		\chi_s (t_1,\ldots,t_k) 			 
		B_{N,N}^{(\beta)}(t) dt
		\nonumber
		\\
		=&  \frac{1}{|A_N|}\int_{\mathcal{W}_k \cap A_N^k}  \chi_s (t_1,
		\ldots,t_k)B_{N,k}^{(\beta)}(t_1,\ldots,t_k) dt_1 \ldots dt_k.
		\label{Darst_E_gamma}
\end{align} 
In order to exploit (\ref{Darst_E_gamma}) we need the following connection 
 between $\sigma_N(H)$ and $\gamma_N(k,H)$, that can be established by combinatorial arguments (see   Corollary 2.4.11, Lemma 2.4.9 and Lemma 2.4.12 in \cite{katzsarnak}). 
\begin{lemma}[chap. 2 in \cite{katzsarnak}]\label{lemma_combinatoric}
For $N \in \mathbb N$, $m \leq N$ and $s \geq 0$ we have
\begin{enumerate} 
\item[(i)] 
\begin{equation}
\label{combinatorics}
 \displaystyle
\int_{0}^s  \, d\sigma_N(H)= \sum_{k=2}^N (-1)^k \int_0^s
\, d \gamma_N(k,H),
\end{equation}
\item[(ii)]  
\begin{equation*}
(-1)^m\int_0^s d\sigma_N(H) \leq (-1)^m\sum_{k=2}^m (-1)^k \int_0^s d\gamma_N(k,H).
\end{equation*}
\end{enumerate}
\end{lemma}
We can now derive the  point-wise convergence of the expected empirical spacing
distribution from  Lemma \ref{lemma_combinatoric}.
This point-wise convergence was derived in  \cite{deift1} for unitary invariant ensembles and in \cite{Alsm_Loewe} for orthogonal and symplectic invariant ensembles.
We specify the error estimates of \cite{Alsm_Loewe} in the proof of the subsequent theorem, where we follow the ideas of \cite{katzsarnak} for unitary circular ensembles.  
\begin{theorem}[c.f.~\cite{katzsarnak}, \cite{deift1}, \cite{Alsm_Loewe}]\label{pointwise}
For $\beta=1,2,4$ we have under
Assumption \ref{main_asssumption}
\begin{enumerate}
\item[(i)]
 \begin{align} \mathbb E_{N,\beta} \left(\int_0^s d\gamma_N(k,H) \right) = &\int_{0 \leq z_2 \leq \ldots \leq z_k \leq s}  W_k^{(\beta)}(0, z_2,\ldots,z_k) 
		dz_2 \ldots dz_k \nonumber \\
		+ &\frac{1}{(k-1)!} s^k \frac{1}{|A_N|} C^k k^{k/2} + s^{k-1}   C^k \mathcal{O}(\kappa_N),\label{absch_lim_gamma}
\end{align}
\item[(ii)]
\begin{equation*} 
		\lim_{N \to \infty} \mathbb E_{N,\beta} \left(  \int_0^s d\sigma_N(H)\right) = \sum_{k\geq 			2} (-1)^k \int_{0 \leq z_2 \leq \ldots \leq z_k \leq s}  W_k^{(\beta)}(0, z_2,		
		\ldots,z_k) dz_2 \ldots dz_k.
	\end{equation*}
In particular, we claim that the series on the right hand side of the equation converges. 
\end{enumerate}
\end{theorem}
\begin{proof} 
The proof is in the spirit of \cite[chap. 5]{katzsarnak}. 
We take expectations in (\ref{combinatorics}) and consider the representation of the expectation of integrals with respect to $\gamma_N$ in (\ref{Darst_E_gamma}).
It is straightforward to prove the following bound
	\begin{equation*} 
		\frac{1}{|A_N|}\int_{\mathcal{W}_k \cap A_N^k}  \chi_s (t_1,\ldots,t_k) 
		dt_1 \ldots dt_k \leq \frac{s^{k-1}}{(k-1)!},
	\end{equation*}
which, together with the uniform convergence of the correlation functions $B_{N,k}^{(\beta)}(t_1,\ldots,t_k)=W_k^{(\beta)}(t_1,\ldots,t_k) + k!   C^k \mathcal{O}(\kappa_N), \, \, t_1,\ldots, t_k \in A_N$ (Lemma \ref{result_corr_func}), leads to 
	\begin{align*}
		 &\frac{1}{|A_N|}\int_{\mathcal{W}_k \cap A_N^k}  \chi_s 				(t_1,
		 \ldots,t_k)B_{N,k}^{(\beta)}(t_1,\ldots,t_k) dt_1 \ldots 		dt_k 
		 \\
		 =&  \frac{1}{|A_N|}\int_{\mathcal{W}_k \cap A_N^k}  \chi_s 			(t_1,
		 \ldots,t_k)W_k^{(\beta)}(t_1,\ldots,t_k) dt_1 \ldots 			dt_k +  
		\frac{s^{k-1}}{(k-1)!}k!   C^k \mathcal{O}(\kappa_N).
	\end{align*}
Using the translational invariance of $W_k^{(\beta)}$ (see Lemma
\ref{result_corr_func}) together with the change of variables $z_1=t_1, z_i=t_i-t_1, i=2,\ldots, k$ and the definition of $\chi_s$, we have
	\begin{align*}
		&\frac{1}{|A_N|}\int_{\mathcal{W}_k \cap A_N^k}  \chi_s (t_1,
		\ldots,t_k)W_k^{(\beta)}(t_1,\ldots,t_k) dt_1 \ldots 			dt_k 
		\\
		=&  \int_{0 \leq z_2 \leq \ldots \leq z_k \leq s}  W_k^{(\beta)}(0, z_2,
		\ldots,z_k) dz_2 \ldots 			dz_k 
		\\
		-&\frac{1}{|A_N|}\int_{A_N} \left(
		\int_{0 \leq z_2 \leq \ldots \leq z_k \leq s}    W_k^{(\beta)}(0, z_2,\ldots,z_k) 
		\left( 1-\prod_{j=2}^k \chi_{A_N} (z_1+z_j)\right) dz_2 \ldots 			dz_k 
		\right)dz_1
		\\ 
		=& \int_{0 \leq z_2 \leq \ldots \leq z_k \leq s}  W_k^{(\beta)}(0, z_2,\ldots,z_k) 
		dz_2 \ldots dz_k + \frac{1}{(k-1)!} s^k C^k k^{k/2}\mathcal{O}\left(\frac{1}{|A_N|}\right) . 
	\end{align*}
This completes the proof of (i).
In order to prove (ii), we note that by 
the upper bounds on $W_k^{(\beta)}$ (Lemma
\ref{result_corr_func}) and  Stirling's formula we have
  \begin{equation} 
  \label{schranke_W}
 \lim_{N \to \infty}  \mathbb E_{N,\beta} \left(\int_0^s d\gamma_N(k,H) \right) \leq 
 C^k s^{k-1} \left( \frac{1}{\sqrt{k-1}} \right)^{k-1} \xrightarrow{k \to \infty} 0.
  \end{equation} 
To verify that we may interchange the limit $N \to \infty$ with the infinite summation over $k$ in (\ref{combinatorics}) after taking expectations, we consider for $m$ odd (Lemma~\ref{lemma_combinatoric}~$(ii)$)
 \begin{equation}
\sum_{k=2}^m (-1)^k    \mathbb E_{N,\beta} \left(\int_0^s d\gamma_N(k,H) \right) \leq       \mathbb E_{N,\beta} \left(  \int_0^s d\sigma_N(H)\right) 
  \leq \sum_{k=2}^{m+1} (-1)^k    \mathbb E_{N,\beta} \left(\int_0^s d\gamma_N(k,H) \right). 
  \label{theorem3:eq2}
  \end{equation} 
On the one hand we take the   limit inferior and on the other hand the  limit superior 
for $N \to \infty$ in (\ref{theorem3:eq2}). Statement (ii) is then a consequence of the convergence for $N \to \infty$ implied by (i) and (\ref{schranke_W}).
\end{proof}

\begin{remark}\label{univ_mu_beta}
So far, we showed that the limit 
\begin{equation*}
F_{\beta}(s) \coloneqq\lim_{N \to \infty}  \mathbb E_{N,\beta} \left(  \int_0^s d\sigma_N(H)\right)
\end{equation*}
exists and is given by (ii) in Theorem \ref{pointwise}. The continuity of $F_{\beta}$ is also obvious from (\ref{schranke_W}) and (\ref{theorem3:eq2}).
Moreover, we have (see (\ref{DEt_Darst_B}) to (\ref{Def_J}))
\begin{equation*}
W_{k}^{(2)}(t_1,
		\ldots,t_k)=\det \left( K_{2}(t_i, t_j)\right)_{1\leq i,j\leq k},
\end{equation*}
\begin{equation*}
W_{k}^{(\beta)}(t_1,
		\ldots,t_k)=\text{Pf}(K_{\beta}(t_i,t_j)_{1\leq i,j\leq k}  J), \quad \beta=1,4.
\end{equation*}
Hence, the limiting expected
spacing distribution $F_{\beta}$ does neither depend on $V$ nor on the point $a$ nor on  the interval $I_N$ and is  hence universal.
  We will prove in Lemma \ref{lemma_limiting_spa} below that there exists a probability measure $\mu_{\beta}$ such that $F_{\beta}(s) = \int_0^s d\mu_{\beta}$.
\end{remark}
We can now derive an important estimate for the proof of the Main Theorem.
\begin{corollary} 
\label{koro_ew}
For any $\alpha>0, N\in\mathbb N$ and $L=L(N) \in \mathbb N$ we have
\begin{align} 
 			  \left| F_{\beta}(\alpha) - \int_{0}^\alpha  \,
 			d\sigma_N(H) \right| 
 			&\leq \sum_{k=2}^L \left|
 			\mathbb{E}_{N,\beta}\left( \int_{0}^\alpha  \, d\gamma_N(k,H) \right) -
 			\int_{0}^\alpha  \, d\gamma_N(k,H) \right|
 			\label{coro_eq1} 
 			\\ 
 			+&\left( \frac{1}{|A_N|}+ \mathcal{O}(\kappa_N)+  \left( \frac{1}{\sqrt{L-1}} \right)^{L-1} \right) (C 
 			\overline{\alpha})^{L+1}.  
 			\label{coro_eq2}
\end{align}			
\end{corollary}
\begin{proof}
We set $ \displaystyle 
E_{\beta}(\alpha,k)\coloneqq  \lim_{N \to \infty}  \mathbb E_{N,\beta} \left(\int_0^{\alpha} d\gamma_N(k,H) \right) $
and observe that we  
have (see   (\ref{theorem3:eq2}) and Lemma \ref{lemma_combinatoric})
 for any odd integers $m_1, m_2$    
	\begin{equation}
	\label{ev7:eq3}
	\sum_{k=2}^{m_1} (-1)^k E_{\beta}(\alpha,k) \leq F_{\beta}(\alpha)\leq \sum_{k=2}^{m_2+1} (-1)^k E_{\beta}(\alpha,k), 
	\end{equation}
	\begin{equation}
	\label{ev7:eq4}
	-\sum_{n=2}^{m_1-1} (-1)^n \int_{0}^{\alpha}  \, d \gamma_N(n,H) \leq
	-\int_{0}^\alpha \, d \sigma_N(H) \leq
	  	-\sum_{n=2}^{m_2} (-1)^n \int_{0}^\alpha  \, d \gamma_N(n,H).
	\end{equation}
	For $L\in \mathbb{N}$ we set $m_1=L$, $m_2=L$ if $L$ is odd and 
	   $m_1=L+1$, $m_2=L-1$ if $L$ is even. 
Adding inequalities (\ref{ev7:eq3}) and (\ref{ev7:eq4}) and using the triangle inequality
 yields
 \begin{align*} 
 			&  \left|  F_{\beta}(\alpha) - \int_{0}^\alpha  \,
 			d\sigma_N(H) \right| \nonumber  
 			\leq \sum_{k=2}^L \left|
 			\mathbb{E}_{N,\beta}\left( \int_{0}^\alpha  \, d\gamma_N(k,H) \right) -
 			\int_{0}^\alpha  \, d\gamma_N(k,H) \right| \nonumber 
 			\\ 
 			+& \sum_{k=2}^L
 			\left| \mathbb{E}_{N,\beta} \left( \int_{0}^\alpha  \, d\gamma_N(k,H) \right)
 			- E_{\beta}(\alpha,k) \right| + E_{\beta}(\alpha,L)+E_{\beta}(\alpha,L+1). 
 		\end{align*}
We can further estimate the right hand side of the above inequality using (\ref{absch_lim_gamma}) and (\ref{schranke_W}) together with the geometric series and the boundedness of $k^{k/2}/{(k-1)!}$ (by Stirling's formula).  
\end{proof}
\begin{remark}\label{hinweis_var}
We recall that in order to prove the Main Theorem (Theorem \ref{MAINTHEO}) we
need to provide an estimate on $\mathbb{E}_{N,\beta} \left(\sup_{s \in
\mathbb{R}} \left|\int_{0}^s \, d\sigma_N(H) - F_{\beta}(s)\right|\right)$. Except for the supremum, we can take
expectations in Corollary \ref{koro_ew}. Then the Cauchy-Schwarz inequality suggests 
to provide an appropriate estimate on the variance of $\int_0^{s} d\gamma_N(k,H)$, which is at the heart of the proof of the Main Theorem and is carried out in Section \ref{Main_est}. \end{remark}
We end this section by some results about the limiting spacing distributions $F_{\beta}$, which will be needed in Section \ref{sec:completing}. 

\begin{lemma}[for $\beta=2$ c.f.~\cite{katzsarnak}]\label{lemma_limiting_spa}
Let Assumption  \ref{main_asssumption} hold for $\beta=1,2,4$.
Then there exist  probability  measures $\mu_{\beta}$ on $\mathbb R$  such that 
	\begin{equation}\label{pointwise.1}
	F_{\beta}(s) \coloneqq \lim_{N \to \infty} \mathbb E_{N,\beta} \left(  \int_0^s d\sigma_N(H)\right) = \int_0^s d\mu_{\beta}.
\end{equation}
Moreover, 
there exist constants $A_{\beta} \geq 1, B_{\beta} \geq 0$ such that 
\begin{equation} 
\label{tail_est}
1-\int_0^s d\mu_{\beta} \leq A_{\beta} e^{-B_{\beta}s^2}, \quad s\geq 0.
\end{equation}
\end{lemma}
\begin{proof}
We observe that the total mass of $\sigma_N(H)$ is the number of eigenvalues that lie in $I_N$ reduced by one and normalised by $1/|A_N|$. Hence, we can calculate (by Remark \ref{intout}) that under Assumption \ref{main_asssumption}
\begin{equation}\label{total_mass}
\mathbb E_{N,\beta}\left(   \int_{\mathbb R} \, d\sigma_N(H)\right)=1-\frac{1}{|A_N|}+ \mathcal{O}\left( \kappa_N\right).
\end{equation}
Similar arguments as those that led to the existence of $  \lim_{N\to \infty} \mathbb E_{N,\beta} \left(  \int_0^s d\sigma_N(H)\right)$
then lead to the existence of
\begin{equation*}
\lim_{N \to \infty}\int_{\mathbb R} f \, d\nu_N, \quad \quad  \nu_{N} (A) \coloneqq \frac{1}{\mathbb E_{N,\beta}\left(   \int_{\mathbb R} \, d\sigma_N(H)\right)} \mathbb E_{N,\beta}\left(   \int_A \, d\sigma_N(H)\right)
\end{equation*}
for all continuous functions $f$ that vanish at infinity.
Applying  \cite[Section VIII.1, Theorem 2]{feller} and then \cite[Section VIII.1, Theorem 1]{feller}  to the probability measures $\nu_N$ shows the existence of a limiting measure $\mu_{\beta}$ with $\mu_{\beta} (\mathbb R) \leq 1$
and 
	\begin{equation} \label{pointwise.2}
	F_{\beta}(s)  = \int_0^s d\mu_{\beta} \eqqcolon H_{\beta}(s)
\end{equation}
 for every point of continuity   of $H_{\beta}$.
 To show the continuity of $H_{\beta}$, suppose that   $H_{\beta}$ has a discontinuity at $x_0$  and  let $(y_n)_{n \in \mathbb N}, (z_n)_{n \in \mathbb N}$ denote sequences with $y_n \nearrow x_0$, $z_n \searrow x_0$ as $n$ tends to infinity. Without loss of generality $z_n,y_n$ can be assumed to be points of continuity of $H_{\beta}$, since the  
finite total mass of $\mu_{\beta}$  implies that $H_{\beta}$ may have a countable number of  discontinuities only. By (\ref{pointwise.2}) and the continuity of $F_{\beta}$ we derive that the large $n$ limits of $H_{\beta}(z_n)$ and $H_{\beta}(y_n)$ coincide. By the monotonicity of $H_{\beta}$, these limits have to be equal to $H_{\beta}(x_0)$, proving the continuity of 
$H_{\beta}$. Hence, (\ref{pointwise.2}) is valid for all $s>0.$

In order to show (\ref{tail_est}), we use the connection between the limiting spacing distribution and  gap probabilities, which are given by
\begin{equation*} 
G_{\beta} (s) := \begin{cases} 
\det(1-K_2|_{L^2(0,s)}), & \text{for } \beta =2 
\\ 
\sqrt{\det(1-K_4|_{L^2(0,s) \times L^2(0,s)})}, & \text{for } \beta=4 \\
\sqrt{{\det}_2(1-K_1|_{L^2(0,s) \times L^2(0,s)})}, & \text{for } \beta=1
 \end{cases},
\end{equation*} 
 where ${\det}_2$ denotes the regularised 2-determinant   (see \cite[(4.50)]{deift2}),  
 $K_2|_{L^2(0,s)}$ denotes the integral operator on $L^2|_{(0,s)}$ with kernel $K_{2}$      and  for $\beta=1,4$ 
 the term  $K_{\beta}|_{L^2(0,s) \times L^2(0,s)}$ denotes the operator on $L^2|_{(0,s)} \times L^2 |_{(0,s)}$ with matrix kernel $K_{\beta}$.
The first step towards (\ref{tail_est}) is to show  
\begin{equation} \label{mubeta2}
		G_{\beta}(s) =1+ \sum_{k=1}^{\infty} (-1)^k \int_{0 \leq x_1\leq  \ldots \leq x_{k} \leq s} W_k^{(\beta)}(x_1,\ldots,x_k) \, dx_1 \ldots dx_k.
	\end{equation} 
	For $\beta=2$ equation (\ref{mubeta2}) is obvious from the standard definition of Fredholm determinants. For $\beta=1$ and $\beta=4$, 
heuristically, equation (\ref{mubeta2}) is  obtained as the   limit of the finite $N$ version, which reads  (combining \cite[p. 108/109]{deift1} and \cite[(4.66)]{deift2}) 
\begin{align*}
\mathbb P_{N,\beta} \left( \{\lambda_1,\ldots,\lambda_N\} \cap (0,s) =\emptyset \right)
&= \sum_{k=0}^N (-1)^k \int_{0 \leq x_1\leq  \ldots \leq x_{k} \leq s} B_{N,k}^{(\beta)}(x_1,\ldots,x_k) \, dx_1 \ldots dx_k
\\&=  \sqrt{\det(1-\hat{K}_{N,\beta}|_{L^2(0,s) \times L^2(0,s)})}.
\end{align*}
 This finite $N$ version is well known (see e.g.~\cite{deift1}, \cite{deift2}), but in order to make the  approximation arguments required for (\ref{mubeta2}) rigorous one would need subtle information on the convergence of the matrix kernels. Relation  (\ref{mubeta2}) can instead be proven by analytical arguments (see \cite[Section~7]{diss}). 
 
From (\ref{mubeta2}) we can derive the following relation by a straightforward calculation (see \cite[p.\ 126]{deift1}):
\begin{equation}\label{mubeta3}
-G'_{\beta}(s) =1- \int_{0}^{s} \, d\mu_{\beta}.
\end{equation}
Hence, we will prove (\ref{tail_est}) by deriving estimates on $|G_{\beta}'|$. 
The main ingredient for these estimates is the fact that $G_{\beta}$ can be represented in terms  of a Painlev\'e transcendent as follows:
Let 
 $\sigma(s)$ be given as the solution of the differential equation 
      \begin{equation*}
    			(s \sigma '')^2  + 4 (s \sigma' - \sigma) \, (s \sigma' -\sigma + (\sigma')^2)=0
   	\end{equation*}
with $   
   			\sigma(s) \sim -\frac{s}{\pi}-\left( \frac{s}{\pi} \right)^2 - \frac{s^3}{\pi^3} + \mathcal{O}(s^4)  \text{ for } s \to 0.  $
   					Then, with  the notation 
 $v(s)\coloneqq \frac{\sigma(s)}{s}, v(0)\coloneqq \frac{1}{\pi}$,
 it is well-known (see  \cite{JMMS},   \cite{AGZ},  \cite{forrester_witte}) that 
   	\begin{align*}
			G_2(s)&=\exp \left( \int_0^{\pi s }v(t) dt \right),  \quad 
		G_1(s)= \exp \left(- \frac{1}{2} \int_0^{\pi s } \sqrt{ -v'(t) } dt \right)  \sqrt{G_2(s)}, 		
	 \\
   		G_4(s/2)&=\frac{1}{2} \left( G_1(s) + \frac{G_2(s)}{G_1(s)} \right).  
   \end{align*}
   For large $s$ the following asymptotic expansion of $v$ is available in the literature   (see e.g.~\cite[(1.38)]{deift3}, \cite[Remark 3.6.5]{AGZ})
  \begin{equation}
\label{asymp_v_1}
v(s)= -\frac{s}{4}- \frac{1}{4s} + \mathcal{O}\left(  \frac{1}{s^2}\right), \quad s\to \infty.
\end{equation}     
We write the derivatives of the gap probabilities as 
\begin{align}
G_2'(s)& = G_2(s) v(\pi s) \pi,
\\
G_1'(s) & = G_1(s) \frac{\pi}{2} \left( v(\pi s)  - \sqrt{-v'(\pi s)}\right),
\label{G_1}
\\
G_4'(s) & = G_1'(2s) + \frac{\pi G_2(2s)}{2 G_1(2s)} \left( v(2\pi s) + 
\sqrt{-v'(2\pi s)} \right).
\label{G_4}
\end{align}  
  For $\beta=2$, (\ref{asymp_v_1}) leads to 
  \begin{equation}
 \label{est_G2} 
  G_2(s) \leq C e^{-\frac{\pi^2}{8} s^2} s^{-\frac{1}{4}}, \quad \quad  |G_2'(s)|\leq C s^{\frac{3}{4}}e^{-\frac{\pi^2}{8}s^2},
 \end{equation} 
 completing the proof of (\ref{tail_est}) for $\beta=2$.

To derive estimates on $G_1'$ and $G_4'$, we obtain a relation between $v$ and its derivative first. We insert (\ref{G_1}) into (\ref{G_4}) to obtain an expression for $G_4'$ and then  rearrange  $G_4' \leq 0$ (see (\ref{mubeta3})) to
 \begin{equation*}
\sqrt{-v'(2 \pi s )} \leq \frac{1+ H(2s)^{-2}}{1-H(2s)^{-2}} (-v(2 \pi s))
 \end{equation*} 
 with 
 \begin{equation*}
  H(s) \coloneqq \exp \left( \frac{1}{2} \int_0^{\pi s} \sqrt{-v'(t)} dt \right)=\frac{\sqrt{G_2(s)}}{G_1(s)}.
 \end{equation*}
 We obtain the following two estimates:
 \begin{enumerate}
 \item[(i)]
 By the behaviour of $v(s)$ for small $s$, we have  $v'(0)=-\frac{1}{\pi^2}$ and hence $H(s)>1$ for $s>0$. Moreover, $H$ is increasing by definition. Thus, $ \frac{1+ H(2s)^{-2}}{1-H(2s)^{-2}} $ is bounded away from $s=0$ and we have 
 \begin{equation}
\label{ineq_v}
 \sqrt{-v'(2 \pi s )} \leq C(-v(2 \pi s)) \leq Cs, \quad s \geq 1. 
 \end{equation} 
 \item[(ii)]
 We have by the Cauchy-Schwarz inequality:
 \begin{equation}
\label{est_H} H(s) \leq \exp \left(  \frac{1}{2} \sqrt{\pi s} \sqrt{v(0) - v(\pi s)} \right)\leq e^{Cs}.
\end{equation}
 \end{enumerate}
  For $\beta=1$, we use $ G_1(s)=\frac{\sqrt{G_2(s)}}{H(s)}$. Then  $H >1$ together with (\ref{est_G2}) implies 
 $
  G_1(s) \leq C e^{- \frac{\pi^2}{16} s^2}.
$
Hence,  for $s \geq 1$ we have
 \begin{equation}
\label{est_G1} 
  |G_1'(s)| \leq Cs e^{- \frac{\pi^2}{16} s^2}.
  \end{equation}
  For $\beta=4$, we consider (\ref{G_4}) and then use (\ref{est_G1}), (\ref{est_G2}), (\ref{est_H}) and (\ref{ineq_v}) to obtain the desired result. This completes the proof of (\ref{tail_est}), which also implies that  $\mu_{\beta}$ is a probability measure.
\end{proof}
%
%
\section{The Main Estimate}\label{Main_est}
		As suggested in Remark \ref{hinweis_var}, the main result of this section is
		an estimate on the variance of $\gamma_N(k,H)$ (see Theorem \ref{lemmavar.1}). The proof of Theorem  \ref{lemmavar.1} is divided into the proofs of some intermediate results (Lemma \ref{VarHilfslemma} and Lemma \ref{var_zwischenerg}) and finally reduced to the proof of the Main Estimate (Lemma \ref{fundamental_est}).
		\begin{theorem} \label{lemmavar.1}
		Under Assumption \ref{main_asssumption} there exists a constant $C\in \mathbb{R}$ such
		that for $\beta \in \{1,2,4 \}$, $2 \leq k \leq N$ 
		and $\alpha>0$ we have $$ \mathbb{V}_{N,\beta}\left( \int_0^{\alpha} d\gamma_N(k,H) \right)\leq 
				\overline{\alpha}^{2k} C^{k} \left( k^k\, \mathcal{O}\left(\frac{1}{|A_N|}\right)
				+\mathcal{O}\left(\kappa_N\right) \right).
				$$			
		\end{theorem}
For later reference, we note the following corollary which can directly be obtained from Theorem \ref{lemmavar.1}.
\begin{corollary}
For $\alpha>0$ and $L>2$ we can calculate (under Assumption \ref{main_asssumption}) for some $C>1$ 
\begin{align*}
\sum_{k=2}^L \sqrt{\mathbb{V}_{N,\beta}\left( \int_0^{\alpha} d\gamma_N(k,H) \right)} 
&\leq\mathcal{O}\left(\frac{1}{\sqrt{|A_N|}}\right) \sum_{k=2}^L 
  \overline{\alpha}^{k} C^{k}\sqrt{(2k)^k}\, 	+\mathcal{O}\left(\sqrt{\kappa_N}\right)\sum_{k=2}^L\overline{\alpha}^{k} C^{k}  
  \\&= \left(\mathcal{O}\left(\frac{1}{\sqrt{|A_N|}}\right)L^{L/2}	+\mathcal{O}\left(\sqrt{\kappa_N}\right)\right)(C\overline{\alpha})^{L+1}.
\end{align*}
\end{corollary}

We start with the introduction of some notation and prove an intermediate result (Lemma \ref{VarHilfslemma}) for the proof of Theorem \ref{lemmavar.1}. 
		%
		\begin{definition} \label{defD}
			For $\beta \in \{1,2,4 \}$, $N \in \mathbb N$, $k \in \mathbb{N}$ with  $2k \leq N$ 
			 we set 
			\begin{equation*}
 				D_{N,k}^{(\beta)} (t',t'')\coloneqq B_{N,2k}^{(\beta)} (t',t'')
				-B_{N,k}^{(\beta)} (t') B_{N,k}^{(\beta)} (t''),\quad t',t'' \in \mathbb{R}^{k}.
			\end{equation*}  
	\end{definition}
	%
	\begin{lemma} \label{VarHilfslemma} 
		Under Assumption \ref{main_asssumption} there exists a constant $C\in \mathbb{R}$  such that for
		 $\alpha >0$, $\beta \in \{1,2,4 \}$,  $2k \leq N$ we have 	 
			\begin{align}
				 & \mathbb{V}_{N,\beta}\left( \int_0^{\alpha} d\gamma_N(k,H) \right)
				\leq  \frac{1}{|A_N|} \,\,C^{k}\, \,  k^k \, (\max (1,2
				\alpha))^{2k} 
				\nonumber				
				\\ & \hskip3cm+  \frac{1}{|A_N|^2 (k!)^2} 
				\left| \int_{A_N^{2k}} \chi_{\alpha}(t')\chi_{\alpha}(t'') D_{N,k}^{(\beta)} (t',t'') dt'dt'' 
				\right|.  
				\label{Int_F_k_D_N_k}
			\end{align}
	\end{lemma}
	\begin{proof}
	In order to estimate the variance on the left hand side of (\ref{Int_F_k_D_N_k}), 
		we first evaluate the expected second moment. Using again the symmetry of the
		rescaled correlation functions together with Remark \ref{intout} and some
		combinatorial arguments leads to 
			\begin{align}
			 &\mathbb E_{N,\beta}\left(\left( \int_0^{\alpha} d \gamma_N(k,H) \right)^2\right)=
			  \frac{1}{|A_N|^2 N!} \sum_{l=k}^{2k} \, \, \, \sum_{\substack{T,S \subset \{1,\ldots,N\}\\ |T|=|S|=k 
			  \\   \#(T\cup S)=l}} \int_{\mathbb R^N} \chi_{\alpha}(t_T)\chi_{\alpha}(t_S) B_{N,N}^{(\beta)}(t) dt
			\nonumber 
			\\
			&=\frac{1}{|A_N|^2} \sum_{l=k}^{2k} \frac{1}{l!} \binom{l}{k}\binom{k}{l-k}
			\int_{A_N^l} \chi_{\alpha}(t_1,\ldots, t_k) \chi_{\alpha}(t_1,\ldots, t_{2k-l},t_{k+1},\ldots, t_l) 
				B_{N,l}^{(\beta)}(t)dt. 
				\label{eq_red_step1.1}
			\end{align}
In the last step we used that there are $\binom{N}{l} \binom{l}{k}\binom{k}{2k-l}$ 
		possibilities for  sets $T,S \subset  \{1,\ldots, N\}$, $|T|=|S|=k$ with $\#(T\cup S)=l$ and once we decided on $T$ and $S$ we may by symmetry assume that in (\ref{eq_red_step1.1}) 
		$$T=\{1, \ldots, k \}, \quad  S=\{1,\ldots, 2k-l,k+1,\ldots, l\},\quad T\cap S=\{1,\ldots, 2k-l \}.$$
For $l<2k$  the two factors $\chi_{\alpha}$ in the integral in
(\ref{eq_red_step1.1}) have at least one common argument, i.e.~all appearing
variables have a maximal deviation from each other of $2 \alpha$. Hence, for $l<2k$ we first sort the $l$ variables in (\ref{eq_red_step1.1}), change variables as usual, i.e.~$s=t_1 \in A_N$, $z_i=t_i-s \in [0,2\alpha], i\neq 1$ 
and unsort the $l-1$ z-variables. Using the bound on the correlation functions provided in Lemma \ref{result_corr_func}, we arrive at 
\begin{align}
			 & \left| \int_{A_N^l} \chi_{\alpha}(t_1,\ldots, t_k) \chi_{\alpha}(t_1,\ldots, t_{2k-l},t_{k+1},\ldots, t_l) 
				B_{N,l}^{(\beta)}(t)dt \right| 
				\nonumber \\
			 &\leq \frac{l!}{(l-1)!} \, |A_N| \, (2\alpha)^{l-1} C^l \, l^{l/2}  
			 =  l \, |A_N| \, (2\alpha)^{l-1} C^l \, l^{l/2}. 
			\label{estimate_einz_Summ}
			\end{align}
		For $k \leq l \leq 2k-1$ we use 
		$
		 l \leq 3(2k-l)! ((l-k)!)^2
		$ (see \cite[(5.11.10)]{katzsarnak})
		together with (\ref{estimate_einz_Summ}) to bound each of the 
		 $k$ terms of the sum in (\ref{eq_red_step1.1}) with $k \leq l \leq 2k-1$ by
		$$ |A_N| \max(1,2 \alpha)^{2k}  C^{k} \, \, (2k)^k $$
		for some  $C>0$. 
		Taking into consideration the last term ($l=2k$) of the sum  in (\ref{eq_red_step1.1})  we obtain
			\begin{eqnarray*}
				\mathbb{E}_{N,\beta}\left( \left( \int_0^{\alpha} d \gamma_N(k, H) \right)^2
				\right) \leq  \frac{1}{|A_N|}\,\, C^{k} \,\,  k^k (\max (1,2
				\alpha))^{2k}  
				\\ 
				+ \frac{1}{|A_N|^2 (k!)^2} \int_{A_N^{2k}} \chi_{\alpha}(t')\chi_{\alpha}(t'') 
				B_{N, 2k}^{(\beta)}(t',t'')dt'dt''. 
			\end{eqnarray*}		
		The representation of the expected value of $\gamma_N$ in (\ref{Darst_E_gamma}) and in particular its square completes the proof. 
	\end{proof}
	With the above lemma the proof of Theorem \ref{lemmavar.1} reduces to the proof of the following estimate (under Assumption \ref{main_asssumption}) for $2k \leq N$:
		\begin{equation}
		 \frac{1}{|A_N|^2 ((k!)^2} \left| \int_{A_N^{2k}} \chi_{\alpha}(t')\chi_{\alpha}(t'')  D_{N,k}^{(\beta)} (t',t'') dt'dt'' \right| 
		 \leq  
		C^{k} \overline{\alpha}^{2k}\left(\mathcal{O}\left( \frac{1}{|A_N|}\right)
			+\mathcal{O}\left(\kappa_N \right) \right).
			\label{reduction_var1}
		\end{equation}
\begin{remark}
In the case $\beta=2$ estimate (\ref{reduction_var1}) can  be shown by using the determinantal structure of the limiting correlation functions and the square integrability of the sine kernel and similar, yet less involved, arguments as we will use in the cases $\beta=1$ and $\beta=4$. For the rest of this section we focus on $\beta=1$ and $\beta=4$. 
\end{remark}
Next, we will estimate the error that arises if on the left hand side of (\ref{reduction_var1}) we replace $ D_{N,k}^{(\beta)}$ by its large $N$ limit.
To this end we introduce the following notation.
\begin{definition}\label{Def_Vnk}
For $\beta=1,4, k \in \mathbb N$  and $t',t'' \in \mathbb R^{k}$ we set 
$$V_k^{(\beta)}(t',t'')\coloneqq  W_{2k}^{(\beta)}(t',t'')-W_{k}^{(\beta)}(t')W_{k}^{(\beta)}(t'')=\lim_{N\to \infty}D_{N,k}^{(\beta)}(t',t''). $$ 
\end{definition} 
\begin{remark}\label{Rep_D}
We observe that $D_{N,k}^{(\beta)}$ is given by  the right hand side of (\ref{expansion_corr}), 
where  we sum for each $m \in \{1, \ldots, 2k \}$ over all partitions of $\{1,\ldots,2k\}$ into non-empty subsets $S_1,\ldots, S_m$. For $D_{N,k}^{(\beta)}$ these sets have the additional property that
there exists an $i_0 \in \{ 1, \ldots, m\}$ such that
\begin{equation}\label{mixing_property}
S_{i_0} \cap \{1, \ldots , k \} \neq \emptyset \text{ \quad and  \quad }  S_{i_0} \cap \{k+1,\ldots , 2k \} \neq \emptyset. 
\end{equation}
 Then $V_{k}^{(\beta)}$ is given by the same expression,  where $\hat{K}_{N,\beta}$ is replaced by $K_{\beta}$ and $\tilde{t}$ by $t$. 
\end{remark}

\begin{lemma}\label{var_zwischenerg}
Under Assumption \ref{main_asssumption} there exists $C>0$ such that for $k \in \mathbb N$, $N \in \mathbb N$, $2k \leq N$ and $\beta \in \{1,4 \}$ we have
		\begin{align*}
			&\left|\int_{A_N^{2k}} \chi_{\alpha}(t')\chi_{\alpha}(t'') D_{N,k}^{(\beta)} (t',t'') dt'dt'' -
			\int_{A_N^{2k}} \chi_{\alpha}(t')\chi_{\alpha}(t'') V_{k}^{(\beta)} (t',t'') dt'dt'' \right| 
			\\
			\leq & \, \, (2k)! \,  C^{k} \, \mathcal{O}(\kappa_N) (2\alpha)^{2k-2} \, |A_N|^2.
 		\end{align*}
\end{lemma}
\begin{proof}
We  use the representation of  $D_{N,k}^{(\beta)}$ 
and $V_{k}^{(\beta)}$ given in Remark \ref{Rep_D}.
By (\ref{est_F.1}) and (\ref{absch_kombi_1}) (where $k$ has to be replaced by $2k$)
we obtain for $t',t'' \in A_N^{k} $   
\begin{equation*}
D_{N,k}^{(\beta)} (t',t'') =V_{k}^{(\beta)} (t',t'') + (2k)! \, C^{k}\,  \mathcal{O}(\kappa_N),
\end{equation*} 
 completing the proof by the usual change of variables.
\end{proof}
It is straightforward that the proof of Theorem \ref{lemmavar.1}  reduces to the proof of 
	\begin{equation*}
		\frac{1}{|A_N|^2 ((k!)^2}
			\left| \int_{A_N^{2k}}  \chi_{\alpha}(t')\chi_{\alpha}(t'') V_{k}^{(\beta)} (t',t'') dt'dt'' \right| 
			=  C^{k} \overline{\alpha}^{2k}\mathcal{O}\left( \frac{1}{|A_N|}\right).
	\end{equation*}
Hence, by the expansion of $ V_{k}^{(\beta)} $ as described in  Remark \ref{Rep_D} and (\ref{absch_kombi_1})  
the proof of Theorem \ref{lemmavar.1} finally reduces  to the proof of the following lemma, which is at the heart of the proof. 
%
%
\begin{lemma}[Main Estimate]\label{fundamental_est}
Let $\alpha>0$. Then there exists a constant $C>0$ such that for all $k \in \mathbb N$, $N \in \mathbb N$, $k \leq N$, $m\in \{1,\ldots,2k\}$, $\beta \in  \{1,4 \}$, a partition 
		$S_1,\ldots, S_m$ of $\{1,\ldots, 2k \}$ 
		 and bijections $\sigma_1\in \mathfrak{S}(S_1),\ldots, \sigma_m \in \mathfrak{S}(S_m)$ we have: 
		\begin{equation}
				 \left|\sum_{d_1,\ldots, d_{2k}=1}^2 		 		
				\int_{A_N^{2k}}  \left(  \prod_{i=1}^m F_{S_i,\sigma_i}^{K_{\beta}} (d_{S_i}, (t',t'')_{S_i})\right) \, \, \chi_{\alpha}(t')\chi_{\alpha}(t'') 
				dt'\, dt'' \right| 
				\leq  \, \,  C^{k} \, \overline{\alpha}^{2k-1} \, 
				\mathcal{O}(|A_N|). \label{fund_est_2}
		\end{equation}
	The constant implicit in the 
	$\mathcal{O}$-term can be chosen uniformly in $\alpha$, $k$, $N$, $m$, $S_1,\ldots, S_m$ and in $\sigma_1,\ldots, \sigma_m$. The terms $  F_{S_i,\sigma_i}^{K_{\beta}} $ are defined in (\ref{Def_F_S}).
\end{lemma}
Before we can prove the Main Estimate, we need to introduce some more notation. For given  $m, \sigma_1, \ldots, \sigma_m$ and $S_1,\ldots, S_m$ 
we can write with $(t',t'')=(t_1,\ldots,t_{2k})$
\begin{equation}
 \prod_{i=1}^m F_{S_i,\sigma_i}^{K_{\beta}} (d_{S_i}, (t',t'')_{S_i}) = \prod_{i=1}^{2k}  f_i(t_{\mu_i},t_{\nu_i}),
 \quad f_i \in  \{ \mathsf{S}_{\beta},\mathsf{D}_{\beta},\mathsf{I}_{\beta}\}
 \label{def_typ1_typ2}
 \end{equation}
with $\{\mu_1,\ldots,\mu_{2k}\}=\{\nu_{1}, \ldots,\nu_{2k}\}=\{1,\ldots, 2k\}$.
Here, we distinguish two types of factors $f_i(t_{\mu},t_{\nu})$:
\begin{itemize}
\item type (i): 	
$  \displaystyle \mu, \nu \in \{1, \ldots, k \}\quad  \text{ or } \quad \mu,\nu \in \{k+1, \ldots, 2k\} 
$
\item type (ii) is divided into two subtypes:
\begin{itemize} \item type (iia):	$
\displaystyle \quad 																	\mu \in \{ 1, \ldots ,k \}, \quad \nu \in \{ k+1,\ldots, 2k \}
														$
		\item type (iib): 
				$ \displaystyle \quad 
														\nu \in \{1, \ldots ,k \}, \quad \mu \in  \{k+1, \ldots, 2k\}.
														$
								\end{itemize}
	\end{itemize}
	We observe that according to Remark \ref{Rep_D}, and in particular (\ref{mixing_property}), the total amount of factors of type (ii) in (\ref{def_typ1_typ2}) is even since there are as many of type (iia) as there are of type (iib). Moreover, there are at least two factors of type (ii).
		We can now prove the Main Estimate in the simpler case $\beta=1$.

\begin{proof}[Proof of the Main Estimate for $\beta=1$]
As there are at least two factors of type (ii) in (\ref{def_typ1_typ2}), we bound all other factors uniformly by some constant $C$ and we estimate
\begin{align*}
\left| \prod_{i=1}^m F_{S_i,\sigma_i}^{K_{\beta}} (d_{S_i}, (t',t'')_{S_i})\right| 
&\leq C^{2k-2} |f(t_{\mu_1},t_{\nu_1})| \, |g(t_{\mu_2},t_{\nu_2})|
\\
& \leq C^{2k-2}  \frac{1}{2}
\left( f^2(t_{\mu_1},t_{\nu_1}) + g^2(t_{\mu_2},t_{\nu_2}) \right)
\end{align*}
for some $f,g \in \{\mathsf{S}_1,\mathsf{D}_1,\mathsf{I}_1\}$ and $\mu_1 \in \{1,\ldots, k\}$, $ \nu_1 \in \{k+1,\ldots, 2k\}$, $\mu_2 \in \{k+1,\ldots, 2k\}$, $\nu_2 \in \{1,\ldots, k\}$ (i.e.~$f$ denotes some factor of type (iia) and $g$ denotes some factor of type (iib)). Because of the square integrability  of $\mathsf{S}_1,\mathsf{D}_1,\mathsf{I}_1$ given in (\ref{eq_square_int.1}), we obtain with the usual change of variables $x=t_{\mu_1},$ $x_i=t_i-t_{\mu_1}, i\neq \mu_1$ and $y=t_{\nu_1}, y_i=t_{i+k}-t_{\nu_1}, i\neq \nu_1-k$
\begin{equation*}
\left|\int_{A_N^{2k}} f^2(t_{\mu_1},t_{\nu_1}) \, \, \chi_{\alpha}(t')\chi_{\alpha}(t'') 
				dt'\, dt'' \right| = \mathcal{O}(|A_N|) (2 \overline{\alpha})^{2k-2}. 
				\end{equation*}
				A similar estimate for $g$ proves the claim.
\end{proof}

\begin{remark}\label{Remark_beta1}
We note that we will be able to apply the same reasoning as in the proof of the Main Estimate for $\beta=1$, whenever there are two square-integrable factors of type (ii) in (\ref{def_typ1_typ2}).
\end{remark}
The proof of the Main Estimate in the case $\beta=4$ is more involved and requires some preparation. 
We will now prove a lemma that will be used to show that certain terms on the left hand side of (\ref{fund_est_2}) for $\beta=4$ cancel each other. 
\begin{lemma}\label{lemma_G}
Let $r \geq 2$. 
For $r>2$ let  $G:\mathbb R^{r}\times \{1,2\}^{r-2} \to \mathbb R$ be given by
\begin{align*}
G((t_1,\ldots, t_r),& (d_1,\ldots,d_{r-2}))\\
&\coloneqq  (K_4(t_1,t_2))_{1,d_1} (K_4(t_2,t_3))_{d_1,d_2} \ldots (K_4(t_{r-1},t_r))_{d_{r-2},2}
\end{align*}
and for $r=2$ let $G:\mathbb R^{2}  \to \mathbb R$ be given by
\begin{equation*}
G(t_1,t_2)\coloneqq \mathsf{D}_4(t_1,t_2),
\end{equation*}
where $K_4$ is the matrix kernel given in (\ref{notation_S}) to (\ref{notation_I}).
Let $f:\mathbb R^{r} \to \mathbb R$ be a function such that $f(t_1,\ldots,t_r)=f(t_{\sigma(1)}, \ldots,t_{\sigma(r)})$ for any permutation $\sigma$ on $\{1, \ldots, r\}$ and let $A \subset \mathbb R$ be an interval.
Then we have 
\begin{equation*}
 \sum_{d_1,\ldots, d_{r-2}=1}^2\int_{A^r} G((t_1,\ldots, t_r), (d_1,\ldots,d_{r-2})) f(t_1,\ldots,t_r) dt_1\ldots dt_r =0, \quad r>2
\end{equation*}
and  
\begin{equation*}
\int_{A^2} G(t_1,t_2) f(t_1,t_2) dt_1 dt_2 =0, \quad r=2.
\end{equation*}
\end{lemma}
\begin{proof}
Let $r>2$. We introduce an equivalence relation on the $(r-2)$-dimensional vectors $(d_1,\ldots,d_{r-2}) \in \{1,2\}^{r-2}$ with equivalence classes of cardinality one or two. We will see that the corresponding  integrals with $d_1,\ldots,d_{r-2}$ in a one-element equivalence class vanish while the others cancel out one another. 
We let $(\cdot)^*: \{1,2\}^{r-2} \to  \{1,2\}^{r-2} $ be given by 
$$(x_1,\ldots,x_{r-2})^* \coloneqq (h(x_{r-2}), \ldots, h(x_1)), \quad \quad h(n) \coloneqq \begin{cases} 1, &n=2 \\ 2, & n=1\end{cases}.$$
This map is obviously an involution defining an equivalence relation with  equivalence classes $\{x, x^*\}$ for $x \in \{1,2\}^{r-2}$.
It suffices to show
\begin{equation*}
\int_{A^r} G(t, (d_1,\ldots,d_{r-2})) f(t) dt = -\int_{A^r} G(t, (d_1,\ldots,d_{r-2})^*) f(t) dt.
\end{equation*}
For $(1,d_1,\ldots, d_{r-2},2) \in \{1,2\}^r$ we observe the following  (see definition of $K_4$ and its entries in (\ref{notation_S}) to (\ref{notation_I})):
\begin{enumerate}
\item[(i)]
Amongst $(1,d_1), (d_1,d_2), \ldots, (d_{r-2},2)$ there is an odd number of pairs that lie  in the set $\{(1,2), (2,1)\}$.
\item[(ii)]
$$ K_4(x,y)_{1,2}=-K_4(y,x)_{1,2}, \quad K_4(x,y)_{2,1}=-K_4(y,x)_{2,1},$$
$$ K_4(x,y)_{1,1}=K_4(y,x)_{1,1}, \quad K_4(x,y)_{2,2}=K_4(y,x)_{2,2}.$$
\item[(iii)]
$\displaystyle K_4(x,y)_{d_id_{i+1}} = K_4(x,y)_{h(d_{i+1}), h(d_i)}$.
\end{enumerate}
Hence, we obtain the following equations:
\begin{align}
&\int_{A^r} G(t, (d_1,\ldots,d_{r-2})) f(t) dt
\nonumber
\\
=& - \int_{A^r} (K_4(t_2,t_1))_{1,d_1} (K_4(t_3,t_2))_{d_1,d_2} \ldots (K_4(t_{r},t_{r-1}))_{d_{r-2},2} f(t) dt\label{cancel_1}
\\
=&- \int_{A^r} (K_4(t_1,t_2))_{d_{r-2},2}\ldots (K_4(t_{r-2},t_{r-1}))_{d_1,d_2} (K_4(t_{r-1},t_r))_{1,d_1}  f(t) dt \label{cancel_2}
\\
=&- \int_{A^r} (K_4(t_1,t_2))_{1,h(d_{r-2})}\ldots (K_4(t_{r-2},t_{r-1}))_{h(d_2),h(d_1)} (K_4(t_{r-1},t_r))_{h(d_1),2}  f(t) dt \label{cancel_3}
\\
=& -\int_{A^r} G(t, (d_1,\ldots,d_{r-2})^*) f(t) dt.
 \nonumber
\end{align}
Indeed, for (\ref{cancel_1}) we used (i) and (ii), for (\ref{cancel_2}) we reversed the order of the factors and relabelled the arguments, which is allowed by the symmetry of $f$ and finally we  used (iii) to obtain (\ref{cancel_3}). 
This completes the proof in the case $r>2$. For $r=2$ the proof is obvious from $G(t_1,t_2)=-G(t_2,t_1)$.
\end{proof}
\begin{remark}\label{bemerkung_lemma_G}
We consider  the expansion of $F_{S,\sigma}^K (d, t_S) $ in (\ref{Def_F_S}) and observe that the product of factors between two $(2,1)$-entries of $K_4$ in  (\ref{Def_F_S}), i.e.~two $\mathsf{I}_4$-factors, can be written in terms of the function $G$ from Lemma  \ref{lemma_G}. We further  observe that two $(2,1)$-entries of $K_4$ may not appear as adjacent factors in (\ref{Def_F_S}). Together with Lemma \ref{lemma_G}, this is the key observation for the following proof for $\beta=4$. 
\end{remark}
We can now finish the proof of the Main Estimate for $\beta=4$. 
 
\begin{proof}[Proof of the Main Estimate for $\beta=4$]
Let $S_1,\ldots, S_m$ and $\sigma_1\in \mathfrak{S}(S_1),\ldots, \sigma_m \in \mathfrak{S}(S_m)$  be fixed.
In the case $\beta=4$  we need to focus on all the factors of type (ii) in (\ref{def_typ1_typ2}) and hence introduce some notation. Recall that with $t'=(t_1,\ldots,t_k), t''=(t_{k+1},\ldots, t_{2k})$ each factor $f_i$ in the expansion 
\begin{equation}\label{exp_TYPI}
 \prod_{i=1}^m F_{S_i,\sigma_i}^{K_{\beta}} (d_{S_i}, (t',t'')_{S_i}) = \prod_{i=1}^{2k}  f_i(t_{\mu_i},t_{\nu_i}),
 \quad f_i \in  \{ \mathsf{S}_{\beta},\mathsf{D}_{\beta},\mathsf{I}_{\beta}\}
 \end{equation}
can be written  (after some appropriate relabelling of the $d_1,\ldots,d_{2k}$) as
\begin{equation}\label{typI}
 f_i(t_{\mu_i},t_{\nu_i})= K_4(t_{\mu_i},t_{\nu_i})_{d_{\mu_i},d_{\nu_i}}, \quad d_{\mu_i},d_{\nu_i} \in \{1,2\}.
\end{equation}
 We observe that the classification of the term in \eqref{typI} into type (i) or type (ii), does not depend on the values of $d_{\mu_i}$ and $d_{\nu_i}$. 
 Suppose that in (\ref{exp_TYPI}) there are $\theta$ factors of type (iia) and consequently $\theta$ factors of type (iib). 
Assume further (after relabelling if necessary) that $f_1(t_{\mu_1},t_{\nu_1}),\ldots, f_{\theta}(t_{\mu_{\theta}},t_{\nu_{\theta}})$ in (\ref{exp_TYPI}) denote the factors of type (iia)  and that $f_{\theta+1}(t_{\varphi_1},t_{\eta_1}),\ldots,$ $f_{2\theta}(t_{\varphi_{\theta}},t_{\eta_{\theta}})$ denote the factors of type (iib)
 for some  $\mu_i,\eta_i \in \{1,\ldots, k\}$, $\nu_i,\varphi_i \in \{k+1,\ldots, 2k\}$, $i=1,\ldots, \theta$.
 We write (suppressing the $d$-dependency in the notation) for $ i=1,\ldots,\theta$
\begin{align}
\mathcal{P}_i(\mu_i,\nu_i,t',t'')  &:=f_i((t',t'')_{\mu_i}, (t',t'')_{\nu_i}),  =K_4(t_{\mu_i},t_{\nu_i})_{d_{\mu_i},d_{\nu_i}}, \quad &&\text{(type (iia))}
\label{intro_pi1}
\\
\mathcal{P}_{\theta+i}(\varphi_i,\eta_i,t',t'')  &:=f_{\theta+i}((t',t'')_{\varphi_i}, 
(t',t'')_{\eta_i}),  =K_4(t_{\varphi_i},t_{\eta_i})_{d_{\varphi_i},d_{\eta_i}}.\quad &&
\text{(type (iib))}
\label{intro_pi2}
\end{align} 
Let $R_{\text{type (i)}}^{(d)}(t',t'')$ denote the product of all factors of type (i) in (\ref{def_typ1_typ2}). Then we can write 
\begin{equation}
 \prod_{i=1}^m F_{S_i,\sigma_i}^{K_{\beta}} (d_{S_i}, (t',t'')_{S_i})
= \prod_{i=1}^\theta \mathcal{P}_i(\mu_i,\nu_i,t',t'')  
			\prod_{i=1}^\theta  \mathcal{P}_{\theta +i}(\varphi_i,\eta_i,t',t'') R_{\text{type (i)}}^{(d)}(t',t'').
			\label{expansion_fund_est}
\end{equation}
Now, we decompose  
\begin{equation}
\label{Def_P_int}
\mathcal{P}_i=\mathcal{P}_i^{\text{(int)}} +\mathcal{P}_i^{\text{(sing)}} 
\end{equation}
and call $\mathcal{P}_i^{\text{(int)}} $ \textit{integrable} as it satisfies some integrability condition (see (\ref{alpha_int_1})). 
 $\mathcal{P}_i^{\text{(sing)}}$  lacks this property and is thus called \textit{singular} and takes values in $\{-\frac{1}{4},0,\frac{1}{4}  \}$ only.
For given $t_{\mu_1}$ and $t_{\nu_1}$ we define this decomposition as follows:
\begin{itemize}
\item
For $1\leq i \leq \theta$ and $j_1 \in \{\mu_1,\ldots,\mu_{\theta}\}$, $j_2 \in \{\nu_1,\ldots, \nu_{\theta}\}$  we set   
\begin{equation*}
\mathcal{P}_i^ {(\text{sing})}(j_1,j_2,t',t'')\coloneqq 
 		\begin{cases}
 			\, \, \, \,0 & \text{ if }\mathcal{P}_i(j_1,j_2,t',t'')
 			\in \{ \mathsf{S}_4(t_{j_1},t_{j_2}),\mathsf{D}_4(t_{j_1},t_{j_2})\}\\
 		   -1/4& \text{ if } \mathcal{P}_i(j_1,j_2,t',t'')=\mathsf{I}_4(t_{j_1},t_{j_2})
 		   \text{ and }  t_{\mu_1}\leq t_{\nu_1} \\
 			\, \, \, \, 1/4&  \text{ if }\mathcal{P}_i(j_1,j_2,t',t'')=\mathsf{I}_4(t_{j_1},t_{j_2})
 		   \text{ and } t_{\mu_1}> t_{\nu_1}.\\
 		\end{cases}
\end{equation*}
\item
For $\theta+1\leq i\leq 2\theta$,  $j_1 \in \{\varphi_1,\ldots,\varphi_{\theta}\}$, $j_2 \in \{\eta_1,\ldots, \eta_{\theta}\}$ we set 
\begin{equation*}
\mathcal{P}_i^ {(\text{sing})}(j_1,j_2,t',t'')\coloneqq 
 		\begin{cases}
 			\, \, \, \,0 & \text{ if }\mathcal{P}_i(j_1,j_2,t',t'')
 			\in \{ \mathsf{S}_4(t_{j_1},t_{j_2}),\mathsf{D}_4(t_{j_1},t_{j_2})\}\\
 		   \, \, \, \,1/4& \text{ if } \mathcal{P}_i(j_1,j_2,t',t'')=\mathsf{I}_4(t_{j_1},t_{j_2})
 		   \text{ and }  t_{\mu_1}< t_{\nu_1} \\
 			 -1/4&  \text{ if } \mathcal{P}_i(j_1,j_2,t',t'')=\mathsf{I}_4(t_{j_1},t_{j_2})
 		   \text{ and } t_{\mu_1} \geq t_{\nu_1}.\\
 		\end{cases}
\end{equation*}
\item 
Moreover, we set 
\begin{equation*}
\mathcal{P}_i^ {(\text{int})}(j_1,j_2,t',t'')\coloneqq \mathcal{P}_i(j_1,j_2,t',t'')-\mathcal{P}_i^ {(\text{sing})}(j_1,j_2,t',t'').
\end{equation*}
\end{itemize}
Obviously, for $t',t'' \in \mathbb{R}^{k}$ the term $\mathcal{P}_i^ {(\text{int})}(j_1,j_2,t',t'')$ is either equal to $\mathcal{P}_i(j_1,j_2,t',t'')$ or differs from it  by $\pm 1/4$.

Let $u 
\in [-\alpha,\alpha]^{k}, u_{\mu_1}=0$,
$v
\in [-\alpha,\alpha]^{k}, v_{\nu_1}=0, \hat{x}=(x,\ldots,x) \in \mathbb R^{k}$, $\hat{y}=(y,\ldots,y) \in \mathbb R^{k}$.
We observe that in the case $\mathcal{P}_i(r,s,t',t'') =\mathsf{I}_4(t_{r},t_{s})$, we have (see (\ref{Def_R})):
\begin{itemize}
\item
 For $1\leq i \leq \theta$ and $r \in \{\mu_1,\ldots,\mu_{\theta}\}$, $s \in \{\nu_1,\ldots, \nu_{\theta}\}$
\begin{equation*}
\mathcal{P}_i^ {(\text{int})}(r,s,\hat{x}+u,\hat{y}+v)=\mathcal{R}_{u_r,v_{s-k}}(x,y).
\end{equation*}
 \item For
  $\theta+1\leq i\leq 2\theta$ and  
$r \in \{\varphi_1,\ldots,\varphi_{\theta}\}$, $s \in \{\eta_1,\ldots, \eta_{\theta} \}$ 
 \begin{equation*}
 \mathcal{P}_i^ {(\text{int})}(r,s,\hat{x}+u,\hat{y}+v)=\mathcal{R}_{v_{r-k},u_s}(y,x).
 \end{equation*}
 \end{itemize}
 Thus, with (\ref{int_R}) we have 
 for some C>0
 \begin{equation}\label{alpha_int_1}
 	\int_\mathbb R  (\mathcal{P}_i^ {(\text{int})}(r,s,			
  			\hat{x}+u,\hat{y}+v))^2  \, 	
  			dx \leq C\overline{\alpha}, 
  			\quad \left|  \int_\mathcal{A} \mathcal{P}_i^ {(\text{int})}	
			(r,s,\hat{x}+u,\hat{y}+v) \, 	
			dx\right|\leq C \, \overline{\alpha}
 \end{equation}
 for all intervals $\mathcal{A}\subset \mathbb R$  and for all $y \in \mathbb R$.

For $\mathcal{P}_i(j_1,j_2,t',t'') =\mathsf{S}_4(t_{j_1},t_{j_2})$ and $\mathcal{P}_i(j_1,j_2,t',t'') =\mathsf{D}_4(t_{j_1},t_{j_2})$ the estimates in (\ref{alpha_int_1})   follow from (\ref{eq_square_int.1}) and (\ref{eq_int.1}).
We will use the estimates in (\ref{alpha_int_1})   after an appropriate change of variables later and proceed with the consideration of (\ref{expansion_fund_est}).

Substituting (\ref{Def_P_int})
into the right hand side of (\ref{expansion_fund_est}) and expanding  the product,  we obtain $4^\theta$ terms for each configuration of $(d_1,\ldots, d_{2k})\in \{1,2\}^{2k}$. Hence, we have for each $(d_1,\ldots, d_{2k}) \in \{1,2\}^{2k}$
\begin{align} 
&\int_{A_N^{2k}}  \left(  \prod_{i=1}^m F_{S_i,\sigma_i}^{K_{\beta}} (d_{S_i}, (t',t'')_{S_i})\right) \, \, \chi_{\alpha}(t')\chi_{\alpha}(t'') 
				dt'\, dt'' 
				\nonumber
\\
&= \sum_{\widetilde{\mathcal{P}}_1 \in \{ \mathcal{P}_1^ {(\text{int})}, \mathcal{P}_1^ {(\text{sing})}  \},\ldots, \atop \widetilde{\mathcal{P}}_{2 \theta} \in \{ \mathcal{P}_{2 \theta}^ {(\text{int})}, \mathcal{P}_{2\theta}^ {(\text{sing})}  \}}			\int_{A_N^{2k}} \prod_{i=1}^\theta \widetilde{\mathcal{P}}_i(\mu_i,\nu_i,t',t'')  
			\prod_{i=1}^\theta  \widetilde{\mathcal{P}}_{\theta +i}(\varphi_i,\eta_i,t',t'')
			R_{\text{type(i)}}^{(d)}(t', t'') 
			\chi_{\alpha}(t')\chi_{\alpha}(t'')dt'dt''.
			\label{main_int_Q_P_3}
\end{align}
We decompose the sum in (\ref{main_int_Q_P_3}) into two partial sums and a single term as follows: 
\begin{itemize}
\item \textbf{partial sum (I):} The sum of those terms in (\ref{main_int_Q_P_3}), where  at least two of the factors $\widetilde{\mathcal{P}}_i, i=1, \ldots, 2\theta$ are integrable.
\item \textbf{partial sum (II):}  The sum of those terms in (\ref{main_int_Q_P_3}), where  exactly one of the factors   $\widetilde{\mathcal{P}}_i, i=1, \ldots, 2\theta$ is integrable. This sum contains $2 \theta$ terms. 
\item \textbf{remainder term:}  There is only a single term left  that is not subsumed into partial sum (I) or partial sum (II): The term, where non of the factors  $\widetilde{\mathcal{P}}_i, i=1, \ldots, 2\theta$ is integrable or in other words all appearing factors are singular.  
\end{itemize}
Then we have for given $d_1,\ldots,d_{2k}$
\begin{align*}
&\int_{A_N^{2k}} \prod_{i=1}^m F_{S_i,\sigma_i}^{K_{\beta}} (d_{S_i}, (t',t'')_{S_i}) 
			\chi_{\alpha}(t')\chi_{\alpha}(t'')dt'dt'' \\
		=&  \text{ partial sum (I)} + \text{partial sum (II)} \\
		  +&\int_{A_N^{2k}} \prod_{i=1}^\theta \mathcal{P}_i^ {(\text{sing})}(\mu_i,\nu_i,t',t'')  
			\prod_{i=1}^\theta  \mathcal{P}_{\theta + i}^ {(\text{sing})}(\varphi_i,\eta_i,t',t'')
			R_{\text{type(i)}}^{(d)}(t', t'') 
			\chi_{\alpha}(t')\chi_{\alpha}(t'')dt'dt''.
\end{align*}
We will show the following two estimates: 
\begin{equation}\label{partialsum1}
\left| \sum_{d_1,\ldots , d_{2k}=1}^2 \text{ partial sum (X)} \right| \leq  C^{k} \, \overline{\alpha}^{2k-1} \, 
				\mathcal{O}(|A_N|), \quad X=I,II
\end{equation}
and  the equation
 \begin{align}
			&\sum_{d_1,\ldots, d_{2k}=1}^2 \int_{A_N^{2k}} \prod_{i=1}^\theta \mathcal{P}_i^{(\text{sing})}(\mu_i,\nu_i,t',t'')  
			\prod_{i=1}^\theta  \mathcal{P}_{\theta +i}^{(\text{sing})}(\varphi_i,\eta_i,t',t'')\nonumber  \\
			&\hskip6cm\times R_{\text{type(i)}}^{(d)}(t', t'') 
			\chi_{\alpha}(t')\chi_{\alpha}(t'')dt'dt'' 
			=0.
			\label{partialsum3}
	\end{align}
Hence, (\ref{partialsum1}) and (\ref{partialsum3}) imply the Main Estimate for $\beta=4$.

 The inequality in (\ref{partialsum1}) 
 will be obtained by estimating each single term in partial sum (I)  and   in partial sum (II), while the proof of equation (\ref{partialsum3}) relies on the fact that some terms for different values of $(d_1,\ldots,d_{2k})$ cancel out (see Lemma \ref{lemma_G}).  
\\
\\
\underline{Proof of (\ref{partialsum1}) for partial sum (I)}: The proof of (\ref{partialsum1}) can be carried out in the same fashion as in the case $\beta=1$ (see Remark \ref{Remark_beta1}), as there are two square-integrable factors (in the sense of the first inequality in (\ref{alpha_int_1})). 
\\
\\
\underline{Proof of (\ref{partialsum1}) for partial sum (II)}:
We consider a term in the sum in (\ref{main_int_Q_P_3}) and let $\widetilde{\mathcal{P}}_v$ be integrable for some $v\in \{1, \ldots, \theta \}$ and 
let $\widetilde{\mathcal{P}}_i$ be singular for $i \neq v$. 
It suffices to consider the case that all appearing singular factors are stemming from $\mathsf{I}_4$ because otherwise the integrals vanish. 
Then this term equals 
		\begin{equation}
			(-1)^{\theta} (1/4)^{2\theta-1}\int_{A_N^{2k}}  \text{sgn}(t_{\mu_1}-t_{\nu_1})\mathcal{P}_v^{(\text{int})}(\mu_v,\nu_v,t',t'')  
			R_{\text{type(i)}}^{(d)}(t', t'') 
			\chi_{\alpha}(t')\chi_{\alpha}(t'')dt'dt''. 
			\label{genau_ein_int_Faktor_1}
		\end{equation}
We change variables: 
\begin{align*}
 x_{\mu_1}&=t_{\mu_1} \in A_N,&  x_i=t_{i}-t_{\mu_1} \in [-\alpha,\alpha] \text{ for }& i \in \{1, \ldots, k \}\setminus \{ \mu_1\}, 
\\
	 y_{\nu_1-k}&=t_{\nu_1} \in A_N,& y_i=t_{k+ i}-t_{\nu_1} \in [-\alpha,\alpha]  \text{ for }& i \in  \{ 1, \ldots, k  \} \setminus \{ \nu_1-k\}
\end{align*} 
and observe that after changing variables $t'$ and $t''$ are replaced by   ($\nu_0\coloneqq \nu_1-k$)
\begin{align*}
		X& \coloneqq(x_1+x_{\mu_1}, \, x_2+x_{\mu_1},\ldots, \, x_{\mu_1-1}+x_{\mu_1}, \, x_{\mu_1},\, x_{\mu_1+1}+x_{\mu_1},\ldots, x_{k}+x_{\mu_1})  \in A_N^k,
		\\
		Y& \coloneqq (y_1+y_{\nu_0}, \, y_2+y_{\nu_0},\ldots, \, y_{\nu_0-1}+y_{\nu_0}, \, y_{\nu_0},\, y_{\nu_0+1}+y_{\nu_0},\ldots, y_{k}+y_{\nu_0}) \in A_N^k.
	\end{align*} 	
We recall that $R_{\text{type(i)}}^{(d)}(t', t'')$ is the product of factors of type (i) each evaluated at the difference of two components of $t'$ or $t''$ and hence $R_{\text{type(i)}}^{(d)}(X,Y)$ does  not depend on 
 $x_{\mu_1}$ and $y_{\nu_0}$ 
and may be taken outside the $x_{\mu_1}$- and $y_{\nu_0}$-integration. 
Moreover,  the term $\chi_{\alpha}(X)=\chi_{\alpha}(x_1,x_2,\ldots, x_{\mu_1-1},0,x_{\mu_1+1},\ldots,x_k)$ does also not depend on $x_{\mu_1}$.  
We recall that for given $y_{\nu_0}$ and given $x_i, i\neq \mu_1$ the following inequality is immediate from the the second estimate in (\ref{alpha_int_1}): 
	\begin{equation}
		\left|\int_{\substack{x_{\mu_1} \in A_N \\ x_{\mu_1} \geq y_{\nu_0}}}  
				\mathcal{P}_v^{\text{(int)}}(\mu_v,\nu_v,X,Y)  \left( \prod_{i=1}^{k} \chi_{A_N}(X_i) \right)
			dx_{\mu_1} \right| \leq C \overline{\alpha}.
			\label{alpha_int_I}
	\end{equation}

We continue with the consideration of (\ref{genau_ein_int_Faktor_1}) and first integrate with respect to $x_{\mu_1}$  in the region $x_{\mu_1} \geq y_{\nu_0}$ . 
	With (\ref{alpha_int_I}) and $|R_{\text{type(i)}}^{(d)}|_{\infty} \leq C^{2k-2\theta}$ we obtain  
\begin{align*}
&\Big| \int_{[-\alpha,\alpha]^{2k-2}}\Big(\int_{A_N}\Big(\int_{\substack{x_{\mu_1} \in A_N \\ x_{\mu_1} \geq y_{\nu_0}}}  \text{sgn}(x_{\mu_1}-y_{\nu_0})\widetilde{\mathcal{P}}_v(\mu_v,\nu_v,X,Y)  \left(\prod_{i=1}^{k} \chi_{A_N}(X_i) \right)
			 dx_{\mu_1} \Big) 
			\nonumber\\
		 &\hskip2cm \times \left(\prod_{i=1}^{k} \chi_{A_N}(Y_i) \right)dy_{\nu_0}\Big) \chi_{\alpha}(X) \chi_{\alpha}(Y) R_{\text{type(i)}}^{(d)}(X, Y) \prod_{i\neq \mu_1}dx_i
		\prod_{i \neq \nu_0} d y_i \Big| 
		\nonumber
		 \\
	 \leq &C^{k} \overline{\alpha}^{\, 2k-1} |A_N|. 
	\end{align*}	
The same estimate can be obtained for  $x_{\mu_1} < y_{\nu_0}$. 
	This leads to 
	\begin{align*}
	& \left| \int_{A_N^{2k}}  \text{sgn}(t_{\mu_1}-t_{\nu_1})\mathcal{P}_v^{\text{(int)}}(\mu_v,\nu_v,t',t'')  
			R_{\text{type(i)}}^{(d)}(t', t'') 
			\chi_{\alpha}(t')\chi_{\alpha}(t'')dt'dt''  \right| \\ \nonumber  
			\leq &  \quad C^{k} \overline{\alpha}^{\, 2k-1} |A_N|
	\end{align*}
	in the case that  $\widetilde{\mathcal{P}}_v$ is integrable for some $v\in \{1, \ldots, \theta \}$ and the functions
 $\widetilde{\mathcal{P}}_i$ with $i \neq v$ are singular. 
 
 The  case that the only integrable factor is $\widetilde{\mathcal{P}}_v$
for some $v\in \{\theta +1, \ldots, 2\theta \}$ (i.e.~this factor stems from a factor of type (iib) before expanding) can be treated in a similar fashion, which proves (\ref{partialsum1}) for partial sum (II). 
\\ 
\\
\underline{Proof of (\ref{partialsum3})}:
We will show that the integrals themselves on the left hand side of  (\ref{partialsum3}) vanish for certain values of $d_1,\ldots,d_{2k}$ and pairs of the remaining terms cancel out. 
It suffices to consider such configurations of $(d_1,\ldots,d_{2k})$ in (\ref{partialsum3}) that imply 
that
   $\mathcal{P}_i$  stems from  $\mathsf{I}_4$  for all $i=1,\ldots,2\theta$, because otherwise at least one of the $\mathcal{P}_i^{(\text{sing})}$   equals zero and the claim is trivially true. Hence, we have 
  \begin{equation}
\label{Bed_di}  
  (d_{\mu_i},d_{\nu_i})=(d_{\varphi_i},d_{\eta_i})=(2,1)\quad  \text{for } i=1,\ldots,\theta
\end{equation}  
   in (\ref{intro_pi1}) and (\ref{intro_pi2}). 
  
  We observe that amongst the $2 \theta$ singular factors, $\theta$ factors are equal to $1/4$ (e.g.~for  $t_{\mu_1} > t_{\nu_1}$  those factors of type (iia)) and $\theta$ factors are equal to $-1/4$ (e.g.~for  $t_{\mu_1} > t_{\nu_1}$  those factors of type (iib)).
  However,  we obtain
	\begin{align*}
			&\int_{A_N^{2k}} \prod_{i=1}^\theta \widetilde{\mathcal{P}}_i(\mu_i,\nu_i,t',t'')  
			\prod_{i=1}^\theta  \widetilde{\mathcal{P}}_{\theta +i}(\varphi_i,\eta_i,t',t'')
			R_{\text{type(i)}}^{(d)}(t', t'') 
			\chi_{\alpha}(t')\chi_{\alpha}(t'')dt'dt''
			\\
			=& (-1)^\theta \left( \frac{1}{4}\right)^{2\theta}	\int_{A_N^{2k}} 
			R_{\text{type(i)}}^{(d)}(t', t'') 
			\chi_{\alpha}(t')\chi_{\alpha}(t'')dt'dt''. \nonumber 
	\end{align*}
	We need to show 
	\begin{equation}
	\label{vanish_R}
	\sum_{d_1,\ldots,d_{2k}=1 }^2\int_{A_N^{2k}} 
			R_{\text{type(i)}}^{(d)}(t', t'') 
			\chi_{\alpha}(t')\chi_{\alpha}(t'')dt'dt''=0,
	\end{equation}
	where some of the $d_i$ are already given by  (\ref{Bed_di}).
We recall (see Remark \ref{Rep_D}) that there is a set $S_j$ amongst $S_1,\ldots,S_m$ such that 
\begin{equation}
\label{schnitt_sj}
S_{j} \cap \{1, \ldots , k \} \neq \emptyset \text{ \quad and  \quad }  S_{j} \cap \{k+1,\ldots , 2k \} \neq \emptyset.
\end{equation}
We want to study the contribution of $F_{S_j,\sigma_j}^{K_4} (d_{S_j}, (t',t'')_{S_j})$ to $R_{\text{type(i)}}^{(d)}(t', t'')$, i.e.~the factors of type (i) in $F_{S_j,\sigma_j}^{K_4} (d_{S_j}, (t',t'')_{S_j})$ . 
We suppose that
\begin{equation*}
|S_j|=M, \quad S_j=\{i_1,\ldots, i_M\}, \quad i_1 < \ldots < i_M
\end{equation*} 
 and we set 
 \begin{equation*}
 s_{r}:= \sigma_j(i_{r}),\quad d^{(r)}:= d_{i_{r}}, \quad \quad r=1,\ldots,M.
 \end{equation*}
Then we have (see (\ref{Def_F_S}))
	\begin{align}
&F_{S_j,\sigma_j}^{K_4} (d_{S_j}, (t',t'')_{S_j})\nonumber \\
&= \frac{1}{2M} (K_4(t_{s_1},t_{s_2}))_{d^{(1)} d^{(2)}} (K_4(t_{s_2},t_{s_3}))_{d^{(2)} d^{(3)}}
 \ldots (K_4(t_{s_M},t_{s_1}))_{d^{(M)}d^{(1)}}.
\label{exp_F}
\end{align}
Suppose that 
\begin{align*}
a := \min  \{r \in \{1,\ldots,M\} &:   s_r \leq k <s_{r+1}\,  \text{ or } \, s_{r+1} \leq k <s_r\},
\\
 b:= \min  \{r \in \{a+1,\ldots,M\}  &:s_r \leq k <s_{r+1}\,  \text{ or } \, s_{r+1} \leq k <s_r  \},
\end{align*}
i.e., reading from left to right in (\ref{exp_F}), 
$(K_4(t_{s_{a}},t_{s_{a+1}}))_{d^{(a)} d^{(a+1)}}$
denotes the first  factor of type (ii) in (\ref{exp_F}) and 
$(K_4(t_{s_{b}},t_{s_{b+1}}))_{d^{(b)} d^{(b+1)}}$
is the second factor of type (ii). Notice that these exist because of (\ref{schnitt_sj}).
Moreover, we have
\begin{equation*}
(K_4(t_{s_{a}},t_{s_{a+1}}))_{d^{(a)} d^{(a+1)}}, (K_4(t_{s_{b}},t_{s_{b+1}}))_{d^{(b)} d^{(b+1)}}\in \{\pm 1/4\},
\end{equation*}
\begin{equation*}
(d^{(a)} d^{(a+1)})=(d^{(b)} d^{(b+1)})=(2,1).
\end{equation*}
Hence, the  contribution of $F_{S_j,\sigma_j}^{K_4} (d_{S_j}, (t',t'')_{S_j})$ to $R_{\text{type(i)}}^{(d)}(t', t'')$, i.e.~the factors of type (i) in $F_{S_j,\sigma_j}^{K_4} (d_{S_j}, (t',t'')_{S_j})$, 
includes a sequence
\begin{equation}
\label{contri_R_type1}
(K_4(t_{s_{a+1}},t_{s_{a+2}}))_{1 d^{(a+2)}}(K_4(t_{s_{a+2}},t_{s_{a+3}}))_{d^{(a+2)} d^{(a+3)}} 
  \ldots (K_4(t_{s_{b-1}},t_{s_{b}}))_{d^{(b-1)} 2}. 
\end{equation}
Considering (\ref{contri_R_type1}), we observe the following:
\begin{enumerate}
\item[(i)]
There is at least one factor in (\ref{contri_R_type1}) because  products of form  $\mathsf{I}_4(x,y)\mathsf{I}_4(y,z)$ may not appear in the expansion of $F_{S_j,\sigma_j}^{K_4} (d_{S_j}, (t',t'')_{S_j})$.
\item[(ii)]
For any factor $(K_4(t_{\mu},t_{\nu}))_{d_{\mu'} d_{\nu'}}$ of   $R_{\text{type(i)}}^{(d)}(t', t'')$ that is not listed in (\ref{contri_R_type1})
we have 
\begin{equation*}
 \mu,\nu \notin \{ s_{a+1},s_{a+2}, \ldots, s_{b} \}, \quad  \quad 
  \mu',\nu' \notin \{ d_{i_{a+2}}, d_{i_{a+3}}, \ldots, d_{i_{b-1}}\}.
\end{equation*}
\item[(iii)] We have 
\begin{equation*}
\{ s_{a+1},s_{a+2}, \ldots, s_{b} \} \subset \{1,\ldots, k\} \quad \text{or} \quad \{ s_{a+1},s_{a+2}, \ldots, s_{b} \} \subset \{k+1,\ldots, 2k\}.
\end{equation*}
\item[(iv)]
Using the notation of Lemma \ref{lemma_G} the term in (\ref{contri_R_type1}) equals
\begin{equation*}
G((t_{s_{a+1}},\ldots, t_{ s_{b}}),(d_{i_{a+2}},  \ldots, d_{i_{b-1}})
). 
\end{equation*}
\end{enumerate}
These observations allow us, in order to prove (\ref{vanish_R}), to first integrate the term in (\ref{contri_R_type1}) with respect to $t_{s_{a+1}},\ldots, t_{ s_{b}}$. By Lemma \ref{lemma_G} we have 
\begin{equation*}
\sum_{d_{i_{a+2}},\ldots, d_{i_{b-1}}=1}^2 \int_{A_N^{b-a}} G((t_{s_{a+1}},\ldots, t_{ s_{b}}),(d_{i_{a+2}},\ldots, d_{i_{b-1}})) 
 \chi_{\alpha}(t')\chi_{\alpha}(t'')dt_{s_{a+1}}\ldots dt_{ s_{b}}=0.
\end{equation*}
This shows (\ref{vanish_R}) and completes the proof.
\end{proof}

\section{Completing the Proof of the Main Theorem}\label{sec:completing}In this section we show, again following the ideas of \cite{katzsarnak}, how the estimates of Section \ref{Sec:pointwise} and Section \ref{Main_est} can be combined to prove Theorem \ref{MAINTHEO}.
		For $s >0$  we set
	\begin{equation*} 
		\Delta_N(s,H)\coloneqq \int_0^s \, d\sigma_N(H)-\int_0^s \, d\mu_{\beta}.
	\end{equation*} 
	In order to prove the Main Theorem,   
	we have to provide an upper bound on the expectation 
	$
	 \mathbb{E}_{N,\beta}\left( \sup_{s \in\mathbb{R}}\left|\Delta(s,H)\right|\right)
$ and so far we derived  an upper bound on 
	$  
	 \left|\Delta(s,H)\right|  
$
	for  $s>0$ (Corollary \ref{koro_ew}). In order to deal with the supremum, we follow the suggestions in \cite{katzsarnak} and introduce  nodes as follows.

	For a given matrix size $N\times N$ let $M=M(N) \in \mathbb{N}$ with $\lim_{N \to \infty} M(N)=\infty$ denote the number of considered nodes 
  $0=s_{0}<s_{1}<\ldots < s_{M(N)-1}<s_{M(N)}=\infty $
  such that
	\begin{equation}\label{def_si_nodes}
	    \int_{0}^{s_i}\, d\mu_{\beta}=\frac{i}{M(N)}, \quad i=0,\ldots,M(N).
	\end{equation}
Such nodes $s_i$ exist because of the continuity of $\int_0^s d \mu_{\beta}$ as a function of $s$   and the fact that $\mu_{\beta}$ is a probability measure (see Lemma \ref{lemma_limiting_spa}).
	Observe that the nodes $s_i, i=1,\ldots,M(N)-1$ depend on both $\beta$ and $N$ and we may eventually   
	 stress the $\beta$-dependency of $s_i$ by writing $s_{i,\beta}$ instead. 
	For later use we further introduce the notation
	\begin{equation*}
	\delta_N\coloneqq\max_{\beta=1,2,4}\{1,s_{M(N)-1,\beta}\}. 
	\end{equation*} 
Following further  the ideas of \cite{katzsarnak}, we obtain a relation between $\delta_N$ and $M(N)$ using the tail estimate for $\mu_{\beta}$  as follows.  For $A_{\beta}\geq 1$ and $B_{\beta} \geq 0$ 
	as in  Lemma \ref{lemma_limiting_spa} (ii) and  $C_{\beta}\coloneqq\sqrt{\frac{1+\ln(A_{\beta})}{B_{\beta}}}$ we have for $N$ sufficiently large and hence   $\ln(M(N))>1$
	\begin{equation*}
				   \int_{C_{\beta} \sqrt{\ln(M(N))}}^{\infty} \,  d\mu_{\beta} \leq A_{\beta} e^{-B_{\beta}C_{\beta}^2 \ln(M(N))}
				   = \frac{1}{M(N)}\left( \frac{1}{A_{\beta}}\right)^{\ln(M(N))-1}\leq \frac{1}{M(N)}.			
				\end{equation*}
				By the definition of the nodes $s_i$, this implies $				
				s_{M(N)-1,\beta} \leq  C_{\beta} \sqrt{\ln(M(N))},
$ and hence for some $C>0$
\begin{equation}\label{A_delta}
	   \delta_N\leq C \sqrt{\ln(M(N))}.
	\end{equation}
The following lemma  shows why it is useful to evaluate $|\Delta(s,H)|$ at the nodes $s_1,\ldots, s_{M(N)-1}$ in order to provide an 
upper bound on $|\Delta(s,H)|$ that does not depend on $s$. 
%

%
%
%
\begin{lemma}[Section 3.2 in \cite{katzsarnak}]\label{est_delta_s_H}
   Let $M(N) \in \mathbb{N}$ and $s_{0,\beta},\ldots,s_{M(N),\beta}$ with 
    $\beta \in \{1,2,4 \}$ be as in 
     (\ref{def_si_nodes}). 
    For 
    \begin{equation*}
   \Delta_{M}(H)\coloneqq\max_{i=1,\ldots, M(N)-1}|\Delta(s_{i,\beta},H)|
\end{equation*}
    and $s \geq 0$ we have
      \begin{equation*}
         |\Delta(s,H)|\leq \frac{1}{M(N)} +  \Delta_M(H)+\left| \int_{\mathbb{R}} \, d\sigma_N(H) -1 \right|.
      \end{equation*}
\end{lemma}
As the right hand side of the estimate in Lemma \ref{est_delta_s_H}
does not depend on $s$, we can further estimate $\mathbb{E}_{N,\beta}\left(\sup_{s\in\mathbb{R}}|\Delta(s,H)|\right)$.
\begin{lemma}\label{est_sup}
For  $M=M(N)$ with $M(N) \to \infty$ for $N \to \infty$ and  $s_{0,\beta},\ldots,s_{M(N),\beta}$  as in 
     (\ref{def_si_nodes}) we have 
   \begin{equation*}
      \mathbb{E}_{N,\beta}\left(\sup_{s\in\mathbb{R}}|\Delta(s,H)|\right)\leq        
      \mathbb{E}_{N,\beta}\left(\Delta_M(H)\right)+\frac{1}{M(N)} +  \mathcal{O}(\sqrt{\kappa_N}) + 
      \mathcal{O}\left(\frac{1}{\sqrt{|A_N|}}\right), \quad  \beta \in \{1,2,4 \}
   \end{equation*}
   \end{lemma}
   \begin{proof}
   By Lemma \ref{est_delta_s_H} it remains to show that 
   $$\mathbb E_{N,\beta} \left(\left| \int_{\mathbb{R}} \, d\sigma_N(H) -1 \right| \right) =\mathcal{O}(\sqrt{\kappa_N}) + 
      \mathcal{O}\left(\frac{1}{\sqrt{|A_N|}}\right) .$$ 
      By Jensen's inequality we have
\begin{align}\label{jensen}
\left(\mathbb E_{N,\beta} \left| \int_{\mathbb R} d\sigma_N(H) -1\right| \right)^2 &\leq \mathbb E_{N,\beta}  \left( \left(\int_{\mathbb R} d\sigma_N(H) \right)^2  \right) - 2\mathbb E_{N,\beta} \left( \int_{\mathbb R} d\sigma_N(H)   \right)+1.
\end{align}
To calculate the moments of $ \int_{\mathbb R} d\sigma_N(H)$, we set
$$S(I_N,H) := \# \{i: \tilde{\lambda}_i(H) \in I_N\} $$
and hence by definition $\int_{\mathbb R} d \sigma_N(H) =\frac{1}{|A_N|} (S(I_N,H)-1)$. 
We calculate (using the symmetry of $B_{N,N}^{(\beta)}$)
\begin{align*}
\mathbb E_{N,\beta}(S(I_N,H)) &=\frac{1}{N!} \int_{t_1, \ldots, t_N} \left( \sum_{i=1}^N \chi_{A_N} (t_i) \right) B_{N,N}^{(\beta)} (t_1,\ldots,t_N) dt_1 \ldots dt_N 
\\
&= \int_{A_N}  B_{N,1}^{(\beta)} (t) \, dt = |A_N| (1+\mathcal{O}\left(\kappa_N\right))
\end{align*}
and 
\begin{align}\label{est_sigma_1}
\mathbb E_{N,\beta}\left( \int_{\mathbb R} d \sigma_N(H)\right)=1+\mathcal{O}\left( \kappa_N\right)+ \mathcal{O}\left( \frac{1}{|A_N|}\right).
\end{align}  
In a similar fashion we obtain 
 \begin{align*}
&\mathbb E_{N,\beta}(S(I_N,H)^2) 
= \frac{1}{N!}\int_{t_1,\ldots, \leq t_N} \left(\sum_{i,j=1}^N \chi_{A_N} (t_i)\chi_{A_N} (t_j)\right)  B_{N,N}^{(\beta)} (t_1,\ldots,t_N) dt_1 \ldots dt_N \\
&=|A_N| (1+ \mathcal{O}\left(\kappa_N\right)) +  \frac{N(N-1)}{N!} \int_{t_1,  \ldots,  t_N}  \chi_{A_N} (t_1)\chi_{A_N} (t_2)
 B_{N,N}^{(\beta)} (t_1,\ldots,t_N) dt_1 \ldots dt_N\\
&= |A_N| (1+ \mathcal{O}\left( \kappa_N\right)) +  \int_{A_N^2}  W_2^{(\beta)} (t_1,t_2) \, dt_1 dt_2 + |A_N|^2\mathcal{O}\left( \kappa_N\right).
\end{align*}
In order to show
\begin{equation}\label{est_int_W}
 \int_{A_N^2}  W_2^{(\beta)} (t_1,t_2) \, dt_1 dt_2 = |A_N|^2 + \mathcal{O}(|A_N|),
\end{equation}
we recall that for  $\beta=2$ we have $W_2^{(2)}(t_1,t_2)=1-K_2(t_1,t_2)^2$ and for $\beta=1,4$ we can expand $W_2^{(\beta)}$ via (\ref{expansion_corr_limit}). The square-integrability of the sine-kernel together with 
$$\int_{A_N^2} \mathsf{I}_{\beta}(t_1,t_2) \mathsf{D}_{\beta}(t_2,t_1) dt_1\, dt_2 = \mathcal{O}(|A_N|), \quad \beta=1,4 $$
then implies (\ref{est_int_W}). 
Hence, we have
\begin{equation*}
\frac{1}{|A_N|^2}\mathbb E_{N,\beta}(S(I_N,H)^2) = 1+ \mathcal{O}(|A_N|^{-1}) + \mathcal{O}(\kappa_N)  
\end{equation*}
and 
\begin{align}\label{est_sigma_2}
\mathbb E_{N,\beta}\left(\left( \int_{\mathbb R} d \sigma_N(H)\right)^2 \right)
&= \frac{1}{|A_N|^2}\left(\mathbb E_{N,\beta}(S(I_N,H)^2)  - 2\mathbb E_{N,\beta} S(I_N,H) + 1 \right) 
\nonumber
\\
&=1+  \mathcal{O}\left( |A_N|^{-1}\right) + \mathcal{O}(\kappa_N).
\end{align}
Inserting (\ref{est_sigma_1}) and  (\ref{est_sigma_2}) into (\ref{jensen}) completes the proof. 
     \end{proof}
   As $\kappa_N \to 0, |A_N| \to \infty$ for $N \to \infty$ with Lemma \ref{est_sup} the Main Theorem is a consequence of the following lemma.
   \begin{lemma}
   There exists a sequence $(M(N))_N$ with $M(N)\to \infty$ as $N \to \infty$ such that 
   $$\lim_{N\to\infty}\mathbb{E}_{N,\beta}\left(\Delta_{M(N)}(H)\right) \to 0,\quad \beta=1,2,4.$$
   \end{lemma}
   \begin{proof}
   We consider the estimate  from Corollary \ref{koro_ew} for $|\Delta(s_j,H)|$. We obtain an estimate for  $\Delta_M$ by adding up the terms in (\ref{coro_eq1}), now depending on   $s_j$,  for $j=1,\ldots, M-1$ and finally adding  (\ref{coro_eq2}).
    Taking expectations and using the Cauchy-Schwarz inequality and  the estimate on the variance (Theorem \ref{lemmavar.1}) leads to the following inequalities 
   \begin{align}
\mathbb{E}_{N,\beta}\left(\Delta_{M(N)}(H)\right) &\leq    \sum_{j=1}^{M-1} \sum_{k=2}^L \sqrt{\mathbb{V}_{N,\beta}\left( \int_0^{s_j} d\gamma_N(k,H) \right)} 
\nonumber \\
&+ \left( \frac{1}{|A_N|}+ \mathcal{O}(\kappa_N)+  \left( \frac{1}{\sqrt{L-1}} \right)^{L-1} \right) (C \delta_N)^{L+1} 
\nonumber \\
& \leq M(N) 
\left( \frac{1}{\sqrt{|A_N|}} L^{L/2}	+  
\sqrt{\kappa_N} \right) 
 (C_0\delta_N)^{L+1} \label{final.1}
\\
&+ \left( \frac{1}{|A_N|}+  \kappa_N +  \left( \frac{1}{\sqrt{L-1}} \right)^{L-1} \right) (C_0 \delta_N)^{L+1} \label{final.2}
   \end{align} 
   for some $C_0>1$ and  arbitrary $L$.
 Next, we will  provide sufficient choices for the parameters $L(N)$ and $M(N)$, which have so far been arbitrary.
   
    For given $N$  
   we consider the equations
   \begin{equation} \label{def_L}
   x^{x/2}=|A_N|^{1/10}, \quad \quad C_0^{x+1} (x+1)^{(x+1)/4}= \min(|A_N|^{1/10}, \kappa_N^{-1/5}).
   \end{equation}
   For sufficiently large $N$ the solutions are positive and we let $L_N^{(1)}$ denote a solution of the first equation and $L_N^{(2)}$ a solution of the second equation in (\ref{def_L}). Moreover,  let $L(N)$ denote the largest integer such that 
   $$ L(N) \leq \min(L_N^{(1)},L_N^{(2)}). $$
   Once $L(N)$ is fixed, we let $M(N)$ denote the largest integer such that 
   \begin{equation} \label{Def_M}
   M(N) \leq \min \{ |A_N|^{1/5}, \kappa_N^{-1/5}, \exp (\sqrt{L+1}/C_1^2) \},
   \end{equation}
   where $C_1$ denotes a constant such that (\ref{A_delta}) is valid. Then $M(N) \to\infty, L(N) \to \infty$ as $N \to \infty$.
   For the sake of readability we omit the $N$-dependency in the notation of $M(N)$ and $L(N)$ in the following estimates and obtain by straightforward calculations:
   
   \begin{itemize}
   \item $ \displaystyle \delta_N \leq C_1 \sqrt{\ln (M)} \leq (L+1)^{1/4}$ by (\ref{A_delta}) and (\ref{Def_M})
   \\
   \item $\displaystyle (C_0 \delta_N)^{L+1} \leq |A_N|^{1/10}, \quad (C_0 \delta_N)^{L+1} \leq \kappa_N^{-1/5}$ \quad by definition of $L$
\\   
   \item $\displaystyle M L^{L/2} \leq |A_N|^{1/5} |A_N|^{1/10}$  by (\ref{Def_M})  and (\ref{def_L})
\\  
   \item  $\displaystyle M \sqrt{\kappa_N} \leq \kappa_N^{3/10}$ by (\ref{Def_M})
\\  
   \item $ \displaystyle
   \frac{C_0^{L+1} \delta_N^{L+1}}{(L-1)^{\frac{L-1}{2}}}  \leq \frac{ C_0^{L+1} (L+1)^{\frac{L+1}{4}} }{(L-1)^{\frac{L-1}{2}}}
   \leq C^L \frac{ (L+1)^{\frac{L+1}{4}}}{(L+1)^{\frac{L-1}{2}}}
   = C^L (L+1)^{-\frac{L}{4} + \frac{3}{4}} 
   $.
   \end{itemize}
   With the above inequalities we obtain that each term in (\ref{final.1}) and (\ref{final.2}) vanishes as $N \to \infty$. This completes the proof of the Main Theorem (Theorem \ref{MAINTHEO}).
	\end{proof}

\section*{Acknowledgements}
The author would like to thank the unknown referees and Martin Venker for helpful comments and suggestions.

\bibliographystyle{plain}
\bibliography{bibliography}
\end{document}